\documentclass[11pt]{article}
\usepackage[english]{babel}

\usepackage{amssymb} 
\usepackage{mathtools}
\usepackage{amsthm} 
\usepackage{graphicx} 
\usepackage[all]{xy} 
\usepackage{wasysym}
\usepackage{wrapfig} 
\usepackage{xcolor} 
\usepackage{verbatim}

\usepackage[a4paper]{geometry}
\title{The Ambrose-Singer Theorem for general homogeneous manifolds with applications to symplectic geometry}
\date{}

\author{J.L. Carmona Jim\'enez and M. Castrill\'on L\'opez}

\renewcommand{\to}{\longrightarrow}

\newcommand{\R}{{\mathbb{R}}}

\newcommand{\LM}{\mathcal{L}(M)}
\newcommand{\THol}{\widetilde{Hol}}

\newcommand{\T}{\nabla}
\newcommand{\TT}{\tilde{\T}}
\newcommand{\G}{\mathcal{G}}

\newcommand{\GL}{\mathrm{GL}}
\newcommand{\Ad}{\mathrm{Ad}}

\newcommand{\h}{\mathfrak{h}}
\newcommand{\m}{\mathfrak{m}}
\newcommand{\g}{\mathfrak{g}}

\newcommand{\fin}{\hfill{$\square \square$}}

\newcommand{\dt}{\frac{d}{dt} \Big |_{t=0}}

\newcommand{\SC}[1]{\underset{\text{\tiny{{$#1$}}}}{\text{\Large{$\mathfrak{S}$}}}}

\newcommand{\curv}{\tilde{R}}
\newcommand{\tors}{\tilde{T}}

\newcommand{\xp}{\xi^{\perp}}

\newtheorem{proposition}{Proposition}[section]
\newtheorem{theorem}[proposition]{Theorem}
\newtheorem{definition}[proposition]{Definition}
\newtheorem{corollary}[proposition]{Corollary}
\newtheorem{lemma}[proposition]{Lemma}
\newtheorem{remark}[proposition]{Remark}
\newtheorem{example}[proposition]{Example}

\begin{document}

\maketitle

\begin{abstract}

	The main result of this article provides a characterization of reductive homogeneous spaces equipped with some geometric structure (not necessarily pseudo-Riemannian) in terms of the existence of a certain connection.  The result generalizes the well-known result of Ambrose and Singer for Riemannian homogeneous spaces, as well as its extensions for other geometries found in the literature. The manifold must be connected and simply connected, the connection has to be complete and has to satisfy a set of geometric partial differential equations. If the connection is not complete or the manifold is not simply-connected, the result provides a characterization of reductive locally homogeneous manifolds. Finally, we use these results in the local framework to classify with explicit expressions reductive locally homogeneous almost symplectic, symplectic and Fedosov manifolds.

\end{abstract}


\section{Introduction} \label{Section I}

Locally symmetric spaces are characterized, as it is well known since \'Elie Cartan \cite{C1929}, either by the existence of local geodesic involutions or by having parallel Riemann curvature tensor. The global version of this classical result requires some conditions on the topology of the manifold: connectedness, simply connectedness and completeness. Recall that, in this global version, symmetric spaces become a special type of homogeneous Riemannian manifold. In 1958, Ambrose-Singer \cite{AS1958} generalized the result for an arbitrary homogeneous manifold, still assuming the same topological conditions:
\begin{theorem} \label{Theorem I1}
	Let $(M,g)$ be a connected and simply-connected complete Riemannian manifold. Then, the following statements are equivalent:
	\begin{enumerate}
		\item The manifold $M$ is Riemannian homogeneous.
		\item The manifold $M$ admits a linear connection $\TT$ satisfying:
		\begin{equation*}
		\TT R = 0,\quad \TT S = 0,\quad \TT g = 0,
		\end{equation*}
	\end{enumerate}
where $R$ is the curvature tensor of the Levi-Civita connection $\T^{LC}$ and  $S = \T^{LC} - \TT$.
\end{theorem}

In fact, the classification of the tensor $S$ (also known as \emph{homogeneous structure tensor}) into $O(n)$-irreducible classes explicitly, $n=\mathrm{dim}M$, provides interesting and powerful geometric results taking benefit from the interplay between Partial differential equations, Algebra and Geometry, an ambitious program that formally started with \cite{TV1983} (for a recent reference giving a panoramic view of most of these geometric results, the reader can go to \cite{CC2019}). It is important to note that if $M$ is not simply-connected or complete, the existence of $\TT$ is still extremely useful, as it characterizes locally homogeneous manifolds, a kind of spaces that are more than a mere local version of global spaces (there are locally homogeneous spaces that are not locally isometric to global homogeneous spaces).

Important extensions of the Theorem of Ambrose-Singer have been carried out in the literature. For example, the characterization of (local) homogeneity on pseudo-Riemannian manifolds was developed in \cite{GO1997}. This situation shows a relevant difference with the original Riemannian work since the existence of the metric connection with parallel torsion and curvature characterizes (locally) homogeneous spaces of reductive type only. As we know, the Lie algebra of the group acting transitively on reductive spaces can be decomposed into two factors, invariant under the adjoint action of the isotropy subgroup. Since every Riemannian homogeneous manifold is automatically reductive, this particularity only shows up when dealing with metrics with signature. The second main extension of the homogeneous structure tensors was given when additional geometric structures are considered together with the pseudo-Riemannian metric, see \cite{K1980} and \cite{L2014}. With geometric structure they mean a reduction of the orthogonal frame bundle, that is, a $G$-structure, for a subgroup $G$ of the orthogonal group of the corresponding signature. This reduction is understood to be determined by the existence of a tensor or set of tensors on the manifold characterizing the frames of the corresponding reduction. From that point of view, the group $G$ is the stabilizer of a canonical tensor (or set of canonical tensors) on $\mathbb{R}^n$ by the natural action of $O(p,q),\, p+q=n$. When this geometry structure is included in the picture, the notion of homogeneous spaces requires the transitive action of an isometry group that also conserves the geometric tensors on $M$. Important instances of this situation include K\"ahler, quaternion-K\"ahler, Sasaki or $G_2$ spaces among others. The generalization now requires the existence of a metric connection making parallel curvature and the torsion and the tensors defining the geometry (see \cite{K1980}).

The main goal of this article is the presentation of a complete generalization of all these results in the case of homogeneous spaces in a broad sense, that is, independently of the presence of a pseudo-Riemannian metric on the manifold (see Theorem \ref{Theorem AG2} below). More specifically, we here give a characterization of reductive and homogeneous spaces equipped with a structure defined by a tensor (or a set of tensors) not necessarily associated to a $G$-structure, through the existence of a complete connection satisfying certain conditions of the Ambrose-Singer type. With homogeneity, we understand that a Lie group acts transitively and leaves the tensors invariant. For the local version of the results, we can drop again the topological conditions on the manifold as well as the completeness of the connection to have only the so-called notion of AS-manifold. In that case, reductivity must be defined carefully (in particular, we follow some ideas in \cite{L2015}) and we show that every reductive locally homogeneous manifold in the broad sense can be equipped with an Ambrose Singer connection. As particular instances of our result, if one of the tensors is a pseudo-Riemannian metric, we recover all the traditional Ambrose-Singer theorems in the literature.

Since all these previous characterizations live in the realm of pseudo-Riemannian geometry, the manifold is always equipped with a background connection. Thus, having in mind the affine structure of the space of all linear connections, the AS-connections can be regarded as $(1,1)$-tensors called \emph{homogeneous structure tensors}. From this starting point, we apply our main result to the case where the manifold is also endowed with an additional arbitrary linear connection for which the (local) transformations are assumed to be affine. We thus obtain a result generalizing the notion of homogeneous structure tensors to not metric situations. This line of thought has been followed, from an infinitesimal point of view, in \cite{O1998} and, recently, \cite{BT2020} where some non-metric homogeneous spaces with connection are tackled.

One of the main goals of these new results is to aim at following the fruitful philosophy exploited so far in the literature where explicit geometric results are provided in terms of the classifications of homogeneous structure tensors (in particular, the reader can go to \cite{BGO2011, CC2019} for surveys that collect the main contributions to this topic). Here, we show a similar idea with the classification of the torsion of the Ambrose-Singer connection or, whenever there is another background connection, the corresponding classification of homogeneous structure tensors. The structure and applications of these classifications shape an ambitious project to be developed in future works (almost complex, complex, contact, Poisson, etc). To begin with, in this article, we start this program in the case of (almost) symplectic manifolds, where explicit expressions of the classes of torsion are given. Furthermore, if the manifold is Fedosov (it has a symplectic background connection) the classification is given for the homogeneous structure tensors. The relationship between both points of view is analysed. This instance is purely non-metric and the classical results of \cite{K1980} cannot be applied as we show below.

In all the known cases in the literature, the classifications of homogeneous structure tensors have some classes whose dimension grows linearly with respect to the dimension of the manifold. These are the so-called \emph{classes of linear type}, and they have provided the most interesting results of the geometric characterizations mentioned above (see, again, \cite{BGO2011, CC2019}). Taking inspiration from these facts, we tackle the study of classes of linear type of Fedosov manifolds in the last section of the work. In this case, again, the geometric result is remarkable as it characterizes Hamiltonian (locally) homogeneous foliations of the leaves which are flat with respect to the symplectic background connection. We provide two low dimensional examples, leaving the study of higher dimensional homogeneous Fedosov manifolds for future research.

\section{A generalization of the Ambrose-Singer Theorem}\label{Section AG}

\subsection{The main result}

Let $G$ be a Lie group acting transitively on a smooth manifold $M$. Choosing a point $p_0\in M$, we can identify $M$ with $G/H$ where $H\subset G$ is the isotropy subgroup of $p_0$. Note that $M$ need not be pseudo-Riemannian and $G$ is not necessarily a group of isometries. The manifold is said to be reductive homogeneous if there is a Lie algebra decomposition $\g= \h  \oplus \mathfrak{m}$ for certain vector subspace $\mathfrak{m}\subset \g$ such that $\mathrm{Ad}_h (\m) = \m$, $\forall h \in H$. In this case, the subspace $\mathfrak{m}$ can be identified with $T_{p_0}M$ through the map $\mathfrak{m}\to T_{p_0}M$, $X\mapsto \dt \mathrm{exp}(tX) p_0$.

The action of $G$ on $M$ naturally lifts to the frame bundle $\LM$. It is well known that there is an unique connection in $\LM$, that is, an unique linear connection $\TT$, such that for every reference $u$ at $p\in M$ and for each $X \in \m$, the orbit $\exp(tX)\cdot u$ is horizontal. This is called the \emph{canonical connection} of the reductive decomposition $\g= \h  \oplus \mathfrak{m}$. This connection satisfies the following important result.

\begin{proposition}[Prop. 1.4.15, \cite{CC2019}]	\label{Proposition AG1}
	Let $M=G/H$ be a reductive homogeneous manifold equipped with the canonical connection $\TT$ and let $K$ be an invariant tensor field on $M$ with respect to the action of $G$. Then $\TT K = 0$.
\end{proposition}

In this work, a $n$-dimensional manifold $M$ with a \emph{geometric structure} is understood as a manifold equipped with a tensor or a set of tensors $P_1,\ldots,P_r$, $r\in\mathbb{N}$. This definition is initially more relaxed than the classical notion of geometric structure in the literature (see for example \cite{ML2004}). More precisely, a traditional approach defines a geometric structure as a reduction of the frame bundle through a canonical model linear tensor $P_0\in (\otimes ^s (\mathbb{R}^n)^*)\otimes (\otimes ^l \mathbb{R}^n)$ in $\mathbb{R}^n$. If $L$ is its stabilizer by the natural action of $ Gl(n,\mathbb{R})$ on $(\otimes ^r (\mathbb{R}^n)^*)\otimes (\otimes ^l \mathbb{R}^n)$, a $(r,l)$-tensor $P$ on $M$ defines a traditional geometric structure with model $P_0$ if the map
\begin{equation*}
	k:\mathcal{L}(M)\to (\otimes ^s (\mathbb{R}^n)^*)\otimes (\otimes ^l \mathbb{R}^n)
\end{equation*}
is defined as
\begin{equation*}
k(u)(v_1,...,v_s,\, \alpha _1,...,\alpha _l)=P(u(v_1),...,u(v_s),\, (u^*)^{-1}(\alpha_1),...,(u^*)^{-1}(\alpha_l)),
\end{equation*}
takes values in the $Gl(n,\mathbb{R})$-orbit of $P_0$. In particular, the subset $Q=k^{-1}(P_0)\subset \mathcal{L}(M)$ is a $L$-reduction of the frame bundle.

Essential examples of this situation cover the (pseudo-)Riemannian, K\"ahler, complex, symplectic or Poisson manifolds, among others. Note that some of these examples are metric, in the sense that one of the tensors $P_i$ is a (pseudo)-Riemannian metric, but some other instances are  non-metric. The Poisson case shows a geometric structure that is not necessarily traditional since the bivector tensor associated to a Poisson structure may have singularities that are incompatible with the existence of a model linear tensor $P_0$.

\begin{theorem}\label{Theorem AG2}
	Let $M$ be a connected and simply-connected manifold and let $P_1, ..., P_r$ tensor fields defining a geometric structure on $M$. Then, the following statements are equivalent:
	\begin{enumerate}
		\item The manifold $M=G/H$ is reductive homogeneous with $G$-invariant tensor fields $ P_1, ..., P_r $.
		\item The manifold $M$ admits a complete linear connection $\TT$ satisfying:
		\begin{equation} \label{Equation 1}
			 \TT \curv = 0,\quad \TT \tors = 0,\quad \TT P_i = 0 \quad i = 1, \dots , r
		\end{equation}
	\end{enumerate}
where $\curv$ and $\tors$ are the curvature and torsion tensors of $\TT$.
\end{theorem}

\subsection{Proof of the main result}

Suppose $M=G/H$ is a reductive homogeneous manifold with $G$-invariant tensors fields $ P_1, ..., P_r $. If $ \TT $ is the canonical connection associated to the reductive decomposition, it is well-known that the canonical connection leaves invariant $\curv$ and $\tors$, that is $ \TT \curv = 0,\, \TT \tors = 0$. We also have $\TT P_i = 0,\, i=1,...,r$, from Proposition \ref{Proposition AG1}. The completeness of this connection comes from \cite[Ch. X, Cor. 2.5]{KN1963}.

\noindent Conversely, let $ \TT $ be a complete connection on $M$ satisfying $ \TT \curv = 0,\, \TT \tors = 0, \, \TT P_i = 0,\, i=1,...,r$. We fix a frame $u_0 \in \mathcal{L}(M)$. Let $ (\tilde{P}(u_0) \to M, \widetilde{Hol} (u_0)) $,  $ \tilde{P}(u_0)\subset \mathcal{L}(M)$, be the holonomy bundle of the connection $ \TT $. To simplify the notation, we denote $ \tilde{P}(u_0) $ by $ \tilde{P} $ and the subgroup $ \widetilde{Hol}(u_0) $ by $ \tilde{H} $. We will denote by $ \tilde{\h} $ and $ \h $ the Lie algebras of $\tilde{H}$ and $H$, respectively.

\noindent We now proceed by parts.

\noindent \textbf{A construction of a complete distribution in $ \tilde{P} $:}

\noindent  On one hand, if we choose $ \{A_1, ..., A_m \} $ a basis of $ \tilde{\h} $, the associated fundamental vector fields $ \{A_1^*, ... , A_m^* \} $ in $ \tilde{P} $ are complete.	
On the other hand, for the canonical basis $ \{e_1, ..., e_n \} $ of $\R^n $, the standard vector fields on $ \LM $,
\begin{equation*}
	B (e_1) = B_1, \quad \ldots \quad B (e_n) = B_n.
\end{equation*}
are complete on $\LM$ since $ \TT $ is a complete connection (see \cite[Vol. I, Prop. 6.5, p. 140]{KN1963}). Note that, since $\TT$ restricts to $\tilde{P}$ and each $B_i$ is horizontal with respect to it, these standard vector fields are tangent to $\tilde{P} \subset \LM$. Hence $ \{A^*_1, ..., A^*_m, B_1, ..., B_n \}$ span a complete distribution on $\tilde{P}$.

\noindent  \textbf{The structure coefficients of the generating vectors are constant:}

\noindent We have
\begin{equation*}
[A_k^*, A_l^*] = [A_k, A_l]^*,\quad [A_k^*, B_i] = B(A_k(e_i)).
\end{equation*}
We now check that $[B_i,B_j]$ has constant coefficients. We denote by $\theta$ the contact form on $\LM$ and by $\omega$ the connection form associated to $\TT$. The curvature and torsion of $\omega$ are denoted by $ \Omega $ and $ \Theta $, respectively. Then,
\begin{align*}
\Theta(B_i, B_j) & = - \theta([B_i, B_j]) \in \R^n, \\
\Omega(B_i, B_j) & = - \omega([B_i, B_j]) \in \tilde{\h}.
\end{align*}
Hence, the splitting $[B_i, B_j] = [B_i, B_j]^h+[B_i, B_j]^v$ with respect to $\omega$ can be written as
\begin{equation*}
[B_i, B_j] = B(\theta([B_i,B_j]))+\omega ([B_i, B_j])^*
=-B(\Theta(B_i, B_j)) - (\Omega(B_i, B_j))^*.
\end{equation*}
For every horizontal vector $ \overline{X} \in T_u \tilde{P}$
\begin{align*}
 \overline{X} (\Theta_u(B_i, B_j)) &= u^{- 1} ((\TT_X \tors) (X_i, X_j)) = 0, \\
 \overline{X} (\Omega_u(B_i, B_j)e_k) &= u^{- 1} ((\TT_X \curv) (X_i, X_j, X_k)) = 0,
\end{align*}
where $ X,\, X_i,\, X_j,\, X_k \in T_{\pi(u)} M $ are the projections of $ \overline{X},\, \,B_i,\, B_j,\, B_k $, respectively. Then $ \Theta (B_i, B_j) $ and $ \Omega (B_i, B_j) e_k $ are constants and hence $[B_i, B_j]$ is a combination of $\{A^*_1, ..., A^*_m$, $B_1, ..., B_n \}$ with constant coefficients.

\noindent  \textbf{$\mathbf{M}$ is a homogeneous space.}

Let $ G $ be the universal covering of $ \tilde{P} $ and let $ \rho: G \to \tilde{P} $ be the covering map. The vector fields $ \overline{A_k^*}$ and $\overline{B_i}$ on $G$ projecting to $\overline{A_k^*}$ and $\overline{B_i}$ are complete and the coefficients of the brackets are constant. Hence, (\cite[p. 10, Prop. 1.9]{T1992}), given a chosen point $e \in \rho ^{-1}(u_0)$, we can endow $G$ with a structure of Lie group with neutral element $ e$ and such that the Lie algebra $ \g $ of $G$ is generated by $ \{(\overline{A_k^*})_e, (\overline{B_i})_e \} $. As $ [A_k^*, A_l^*] = [A_k, A_l]^* $, we can consider the Lie subalgebra $ \g_0 \subset \g $ generated by $ \{(\overline {A_k^*})_e \} $ and let $ G_0 \subset G $ be the associated Lie subgroup to $\g _0$.
			
\begin{lemma}\label{Lemma AG3}
	The manifold $M$ is diffeomorphic to $ G / G_0 $ and hence it is homogeneous.
\end{lemma}

\begin{proof}
	The map $\pi_1 = \pi \circ \rho : G \to M$ is a fibration of $M$. We take its exact homotopy sequence:	
\begin{equation*}
	\xymatrix{... \ar[r]
		& \Pi_1(M,y) \ar[r] \ar@{=}[d]
		&  \Pi_0(\pi_1^{-1}(y),b) \ar[r]
		& \Pi_0(G,b) \ar@{=}[d] \\
		& 0
		&
		& 0}
\end{equation*}
where $b\in G$ and $\rho(b) = y$. We infer that $\Pi_0 (\pi_1^{-1}(y), b) = 0$, that is, $\pi_1^{-1}(y)$ is connected. Since $\pi_1$ is continuous, we obtain that it is closed as well. Finally, by the equality $\pi_{1*}(\overline{A_k^*}) = 0$, we can deduce that the fibres are isomorphic to $G_0$ and closed.

\noindent We define
\begin{align*}
p : G/G_0 &\rightarrow M \\
[b] &\longmapsto \pi_1(b).
\end{align*}
This map $p$ is well-defined. Indeed, if we take a fixed point $b_0 \in G_0$ and we express it as $b_0 = \exp(Y_1) ... \exp(Y_s)$, with $Y_1,...,Y_s \in  \{\pi_{1*}(\overline{A_k^*}) = 0\}$, then we have $\pi(b\cdot b_0) = \pi (b)$, for all $b\in G$. Furthermore, $p$ is a diffeomorphism since it is bijective and its differential is a linear isomorphism at each point. The injectivity of the differential can be obtained from the fact that $\pi_1^{-1}(y)$ is isomorphic to $G_0$, where surjectivity is straightforward since $\rho $ and $\pi$ are both surjective.
\end{proof}

\noindent \textbf{The structure tensors are invariant and $\mathbf{M}$ is reductive.}

\begin{lemma}\label{Lemma AG4}
For any $ a \in G $, the lift $ \tilde{L_a}: \mathcal{L}(M)\to \mathcal{L}(M)$ of the map
\begin{align*}
L_a: M & \to M \\
[b] & \longmapsto [a \cdot b]
\end{align*}
restricts to the reduction bundle $\tilde{L}_a: \tilde{P} \to \tilde{P} $.
\end{lemma}

\begin{proof}
Let $ \mathbf{L}_a $ be the left multiplication on $G$ by $ a \in G $. Note that $ L_a \circ \pi_1 = \pi_1 \circ \mathbf{L}_a $. Then
\begin{equation*}
	L_{a *} \circ \pi_{1 *} (\overline{B_i})_b  = \pi_{1 *} \circ \mathbf{L}_{a *} (\overline{B_i})_b  = \pi_{1 *} (\overline{B_i})_{ab},
\end{equation*}
and
\begin{align*}
L_{a *} \rho(b) & = (L_a([b]); L_{a *} \circ \pi_{1 *} (\overline{B_1})_b, ..., L_{a *} \circ \pi_{1*} (\overline{B_n})_b)  \\
& = ([ab]; \pi_{1*} (\overline{B_1})_{ab}, ..., \pi_{1 *} (\overline{B_n})_{ab}) \\
& = \rho(ab) \in \tilde{P}.
\end{align*}
Hence if $ y = \rho(b) $,
\begin{equation*}
	\tilde{L_a} (\rho (b)) = \rho (ab).
\end{equation*}
\end{proof}
\noindent Since $ \tilde{P}$ is included in the reduction of $\mathcal{L}(M)$ defined by the tensors $P_1,...,P_r$, we have that $ \tilde{L_a} $ preserves them.

\noindent On the other hand, $ \tilde {P} $ is a Lie group. The action of $G$ on $ \tilde{P} $ introduced in the previous Lemma \ref{Lemma AG4} is transitive, since $ \tilde{L_a} (\rho (b)) = \rho (ab) $, and also effective because it is constructed by linear transformation. In particular, the Lie algebra of $ \tilde {P} $ is isomorphic to the Lie algebra of its universal covering $ G $, through the isomorphism $ \rho_{e *} $ in the neutral element.

\noindent Finally, we have $ \g = \g_0 \oplus \m $, where $ \m $ is the subspace generated by $ \{B_i \} $, which clearly satisfies $ [\g_0, \m] \subset \m $. Since $ G_0 $ is connected, we have the $\Ad$ invariance of $ \m $, and the proof of Theorem \ref{Theorem AG2} is completed.
	\fin

\begin{remark}\label{Remark AG5}
The case with no geometric structure on $M$ ($r=0$) was treated in \cite[Vol II, Ch. X, Th. 2.8]{KN1963} or \cite{KK1980}. There, the authors characterize connected and simply connected reductive homogeneous manifolds $M=G/H$ by the existence of a complete connection $\TT$ such that $\, \TT \curv = 0,\, \TT \tors = 0$. Theorem \ref{Theorem AG2} thus provides the generalization of this result to manifolds endowed with additional structures ($r\geq 0$).
\end{remark}

\begin{definition} \label{Definition AG6}
Let $(M,P_1,...,P_r)$ be a manifold equipped with a geometric structure defined by a set of tensors $P_1,...,P_r$. A connection $\TT$ is called a \emph{generalized Ambrose-Singer connection} if it satisfies that:
	\begin{equation*}
		\TT \curv = 0,\quad \TT \tors = 0,\quad \TT P_i = 0,\quad i= 1,...,r
	\end{equation*}
where $\curv$ and $\tors$ are the curvature and torsion of $\TT$.
\end{definition}

For short, a generalized Ambrose-Singer connection is called an AS-connection, and the manifold $M$, equipped with the tensors $P_1, ... ,P_r$, is called an AS-manifold.

We note that Theorem \ref{Theorem AG2} generalizes the Ambrose-Singer Theorem \ref{Theorem I1} on Riemannian manifolds $(M,g)$  by setting $r=1$ and $P_1=g$. In this case, the AS conditions \eqref{Equation 1}, that is, $\TT \curv = 0,\, \TT \tors = 0,\, \TT g= 0$,  are known  to be equivalent to the more classical conditions $\TT R = 0,\, \TT S = 0,\, \TT g= 0$, where $S=\TT -\nabla ^{LC}$, and $R$ is the curvature of the Levi-Civita connection $\nabla ^{LC}$. We now show that this equivalence can be analysed from a broader perspective  for manifolds equipped with a fixed connection of which the group acting transitively must be of affine transformations.  More precisely, we have the following result.

\begin{theorem}\label{Theorem AG7}
	Let $M$ be a connected and simply-connected manifold with an affine connection $\T$ and let $P_1, ..., P_r$ tensor fields defining a geometric structure on $M$. Then, the following statements are equivalent:
	\begin{enumerate}
		\item Up to diffeomorphism, $M$ is a reductive homogeneous space $M = G/H$ and $P_1, ... , P_r, \T$ are $G$-invariant.
		\item The manifold $M$ admits a complete linear connection $\TT$ satisfying:
		\begin{equation*}
		\TT R = 0, \quad \TT T = 0,\quad \TT S = 0,\quad \TT P_i = 0 \quad i = 1, ... r,
		\end{equation*}
	\end{enumerate}
	where $R$ and $T$ are the curvature and torsion of $\T$ and $S= \T - \TT$.
\end{theorem}

\begin{proof}
	Let $G\subset \mathrm{Aff}(M, \T)$ be a group acting transitively on $M$ and preserving $P_1, ... , P_r$. Additionally, $G$ preserves the tensor $S = \T - \TT$. Hence, by Theorem \ref{Theorem AG2} we have that
	\begin{equation*}
	\TT \curv = 0,\quad \TT \tors = 0,\quad \TT S = 0,\quad \TT P_i = 0,\quad i= 1,...,r
	\end{equation*}
	which are equivalent to
	\begin{equation*}
	\TT R = 0, \quad \TT T = 0,\quad \TT S = 0,\quad \TT P_i = 0 \quad i = 1, ... , r
	\end{equation*}
	by the following observation,
	\begin{equation*}
	T_{X}Y -\tilde{T}_{X}Y = S_{X}Y -S_{Y}X, \quad  \tilde{R}_{XY} = R_{XY} +  [S_X, S_Y] - S_{S_{X}Y-S_{Y}X}.
	\end{equation*}
	Conversely, by Theorem \ref{Theorem AG2}, we have that there exists a Lie group $G$ preserving $S$, $P_1, ..., P_r$. Every transformation of $G$ on $M$ preserves $S$ and is an affine transformation of $\TT$. Hence, $G$ preserves $S + \TT = \T$ which means they are affine transformations of $\T$.
\end{proof}

\begin{remark}\label{Remark AG8}
	In particular, Theorem \ref{Theorem AG7} covers the case of homogeneous Riemannian manifold when $r=1$, $P_1 = g$ the metric tensor and $\T = \T^{LC}$ is the Levi-Civita connection.
\end{remark}

\begin{definition} \label{Definition AG9}
	Let $(M, P_1,...,P_r)$ be a manifold equipped with a geometric structure defined by a set of tensors $P_1,...,P_r$ together with an affine connection $\nabla$. A homogeneous structure is a collection $(M, P_1,...,P_r,\nabla, \widetilde{\nabla})$ such that
	$$
	\widetilde{\nabla}R = 0, \quad  \widetilde{\nabla}T = 0, \quad \widetilde{\nabla}S = 0, \quad \widetilde{\nabla}P_i = 0, \quad i = 1,...,r,
	$$
	where $S= \nabla -\widetilde{\nabla}$, and $R$ and $T$ are the curvature and torsion of $\nabla$ respectively. For short, we call $S$ a homogeneous structure tensor.
\end{definition}

In particular, homogeneous structure of $\T^{LC}$ are the classical homogeneous structures of Riemannian manifolds.

In the following sections, for the sake of brevity and simplicity, we consider $M$ with a geometrical structure defined by one tensor $K=(P_1, ..., P_r)$, the following results being analogous for a finite set of tensors $(P_1, ..., P_r)$.


\section{Reductive locally homogeneous manifolds} \label{Section RL}

The conditions involved in Theorem \ref{Theorem AG2} are of three different types. First, there is a group of partial differential equations expressed as the vanishing of some covariant derivatives. Second, the completeness of the AS-connection. And finally, a couple of topological conditions (connectedness and simply-connectedness) of the manifold $M$. Connectedness is not an issue, since one usually works with connected components. With respect to simply connectedness, even though essential, it is a condition that can be implemented by working with the universal cover of the manifold, and then project the structures back to the original space. The projection will probably imply that the space is locally homogeneous only, but locally isomorphic to the global homogeneous cover. The completeness, however entails more delicate information since non-complete AS connections may induce locally homogeneous manifolds that are not locally isomorphic to homogeneous spaces. In the Riemannian case we have the following classical result (see \cite{T1992}).
\begin{theorem}\label{Theorem RL1}
	Let $(M,g)$ be a Riemannian manifold. Then, $(M,g)$ is a locally homogeneous manifold if and only if there exists a linear connection $\TT$ satisfying, $\TT \curv = 0$, $\TT \tors = 0 $ and $\TT g = 0$, where $\curv$ and $\tors$ are the curvature and torsion of $\TT$.
\end{theorem}

However, if one wants to move forward, the generalization to pseudo-Riemannian manifolds with signature implies the understanding of the notion of reductivity in the local case. That construction was recently achieved in \cite{L2015}. We generalize below the definition of reductive locally homogeneous manifolds with not necessarily metric structure, and we also characterize these manifolds through a transitive Lie pseudo-group and an AS-connection.

First, we are going to determine the notation related to this section. Let $M$ be a manifold with a geometric structure defined by a tensor or a set of tensors $K=(P_1,...,P_r)$.

\begin{definition} \label{Definition RL2}
	Let $(M,K)$ a manifold with a geometric structure defined by $K$. A \emph{pseudo-group} $\G$ is a collection of locally diffeomorphisms, $\varphi: U_{\varphi} \to M$, such that:
	\begin{itemize}
		\item Identity: $Id_M  \in \G$.
		\item Inverse: If $\varphi \in \G$, then $\varphi ^{-1} \in \G$.
		\item Restriction: If $\varphi \in \G, \varphi: U \to M$ and $V \subset U$, then $\varphi|_V \in \G$.
		\item Continuation: If $\mathrm{dom}(\varphi) = \bigcup U_k$ and $\varphi|_{U_k} \in \G$, then $\varphi \in \G$.
		\item Composition: If $\mathrm{im}(\varphi) \subset \mathrm{dom}(\psi)$, then $\psi \circ \varphi  \in \G$.
	\end{itemize}
	In addition, we will require that $\G$ leave $K$ invariant, that is, for every $\varphi \in \G$, we have that $\varphi ^* K = K$.
\end{definition}
A space endow with geometric structure $(M,K)$ is a \emph{locally homogeneous manifold} if there exists a Lie pseudo-group $\G$ acting transitively on $M$. In order to define a reductive locally homogeneous manifold, we have to know:
\begin{itemize}
	\item  The meaning of isotropy representation related to pseudo-groups.
	\item  The meaning of adjoint function.
\end{itemize}

We again fix a frame $u_0\in \mathcal{L}(M)$ over $p_0\in M$. We define $\G(p_0)$ as the set of transformations for which $p_0$ belongs to the domain of $\varphi$ and $\G(p_0,p_0) \subset \G (p_0)$ the set of transformations such that $\varphi (p_0)=p_0$. The quotient $H(p_0) = \G(p_0, p_0) / \sim$ with respect to the relation $\varphi \sim \psi \iff \varphi |_U = \psi |_U$ for some neighbourhood $U$ of $p_0$, is a topological group. We say that the action of $\G$ on $M$ is \emph{effective} and \emph{closed} if the map
\begin{equation}\label{Equation 2}
	\begin{alignedat}{3}
H(p_0) &\to  \GL(n,\R) \\
\varphi &\longmapsto u_0^{-1} \circ \varphi_* \circ u_0
	\end{alignedat}
\end{equation}
is a monomorphism and the image $\mathbf{H}(u_0)$ is closed, respectively, in particular, $\mathbf{H}(u_0)$ is a Lie subgroup of $\GL(n,\R)$. The morphism \eqref{Equation 2} will be called the isotropy representation of $\G$ on $(M,K)$.

\begin{proposition}\label{Proposition RL3}
	The action of $\G$ on $M$ is effective if and only if for every $\varphi, \psi \in \G$ such that $\varphi(p_0) = \psi(p_0)$ and  $\varphi_{*,p_0} = \psi_{*,p_0}$, then $\varphi = \psi$ in an open neighbourhood of $p_0$.
\end{proposition}

\begin{proof}
	It is obvious that if we have the second condition then we have an effective action.
	
	Conversely if $\varphi, \psi \in \G$ are such that $\varphi(p_0) = \psi(p_0)$ and  $\varphi_{*,p_0} = \psi_{*,p_0}$, we have $\psi \circ \varphi ^{-1}\in H(p_0), \, \psi \circ \varphi ^{-1}(p_0) = p_0$ and $(\psi \circ \varphi ^{-1})_{*,p_0} = Id_{T_{p_0}M}$. Then, $\psi \circ \varphi ^{-1}= Id_M$ in an open neighbourhood of $p_0$.
\end{proof}

We now consider
\begin{equation} \label{Equation 3}
P(u_0): = \{ \varphi_* \circ u_0 : \R^n \to T_{f(p_0)}M : \varphi\in \G(p_0) \}.
\end{equation}
This bundle is a reduction of $(\LM \to M, \GL(n,\R))$ to the group $\mathbf{H}(u_0)$.

\begin{proposition}\label{Proposition RL4}
	If $u_0,u_1\in \mathcal{L}(M)$ are two frames on $p_0$ and $p_1\in M$ respectively, then
	\begin{equation*}
		P(u_1)=P(u_0)g,
	\end{equation*}
	where $g$ is the element in $\GL(n,\R)$ such that $\psi_* u_0 = u_1 g^{-1} $, with $\psi \in \mathcal{G}$, $\varphi (p_0)=p_1$.
\end{proposition}
\begin{proof}
	We define the homomorphism $\sigma: H(p_0)\to H(p_1)$, $\varphi \longmapsto \psi_* \circ \varphi \circ \psi ^{-1}_*$. For the sake of simplicity,
	we also denote by $\sigma: \mathbf{H}(u_0)\to \mathbf{H}(u_1)$ the induced homomorphism by the identification \eqref{Equation 2}. It is a matter of checking that $R_g:\mathcal{L}(M)\to \mathcal{L}(M)$ induces a principal bundle isomorphism between $P(u_0)$ and $P(u_1)$ with associated Lie group homomorphism $\sigma$.
\end{proof}

In particular, the groups $\mathbf{H}(u_0)$ and $\mathbf{H}(u_1)$ are always isomorphic. Because of this, we may simply write $\mathbf{H}$ for any $\mathbf{H}(u_0)$.

Given an element $\varphi \in H(p_0)$, we define
\begin{align*}
\Ad_{\varphi}: T_{u_0}P(u_0) &\to  T_{u_0}P(u_0)  \label{Equation RL3} \\
\dt (\varphi_t)_*(u_0) &\longmapsto \dt (\varphi \circ\varphi_t \circ \varphi^{-1})_* (u_0)
\end{align*}
where $\varphi _t \in \mathcal{G}$,
$t$ belonging to certain interval $(-\epsilon , \epsilon)$.

\begin{definition}\label{Definition RL5}
	Let $(M,K)$ a manifold with a geometric structure. We will say that $(M,K)$ is \emph{reductive locally homogeneous manifold} if there exists a Lie pseudo-group $\G$ acting transitively, effectively and closed on $M$, and we can decompose $T_{u_0}P = \h + \m$, where $\h$ is the lie algebra associated to $H(p_0)$ and $\m $ is a $\Ad(H(p_0))$-invariant subspace.
\end{definition}

The definition depends at first sight on the chosen frame $u_0$. However, this dependence is not real as the following result proves.

\begin{proposition}\label{Proposition RL6}
	Let $u_0$, $u_1 \in \mathcal{L}(M)$ two linear frames. Then $T_{u_1}P(u_1)$ decomposes as $\h + \m_1$ for an $\Ad(H(p_1))$-invariant subspace $\m_1$ if and only $T_{u_0}P(u_0)$ decomposes as $\h + \m_0$ for an $\Ad(H(p_0))$-invariant subspace $\m_0$.
\end{proposition}

\begin{proof}
	Given the decomposition $T_{u_0} P(u_0) = \h + \m_0$ such that $\Ad(H(p_0))_{\varphi} (\m_0) \subset \m_0$, we write $T_{u_1} P(u_1) = \h + \m_1$ with $\m_1 = \Psi_* \m_0$, where $\Psi = R_g \circ \psi_*$ and $\psi\in\mathcal{G}$ is that $\psi_* u_0= u_1 g^{-1}$ and $g\in GL(n,\mathbb{R})$. The subspace $\m_1$ is $\Ad(H(p_1))$-invariant. Indeed, for any element  $X = \dt \varphi_t (u_0)\in \m_0$ and $\varphi \in H(p_0)$, we have that
	\begin{align*}
	&\Psi_* (\Ad(H(p_0))_{\varphi}(X))  =\\
	&= \dt R_g \circ \psi_* \circ \varphi_* \circ (\varphi_t)_* \circ \varphi^{-1}_* (u_0)  \\
	&= \dt \psi_* \circ \varphi_* \circ \psi^{-1}_* \circ \psi_* \circ (\varphi_t)_* \circ \psi^{-1}_* \circ \psi_* \circ \varphi^{-1}_* \circ \psi^{-1}_* \circ u_1  \\
	&=  \Ad(H(p_1))_{(\psi_* \circ \varphi \circ \psi^{-1}_* )}(\dt \psi_* \circ \varphi_t \circ \psi^{-1}_* ( \psi_* u_0g))   \\
	&=  \Ad(H(p_1))_{(\psi_* \circ \varphi \circ \psi^{-1}_* )}(\dt R_g \circ \psi_* \circ \varphi_t ( u_0))   \\
	&=  \Ad(H(p_1))_{\sigma(\varphi)}(\Psi_*(X)).
	\end{align*}	
\end{proof}

Now we give a local version of Theorem \ref{Theorem AG2} above. Furthermore, it provides a generalization of the Tricerri's result Theorem \ref{Theorem RL1}.

\begin{theorem} \label{Theorem RL7}
	Let $(M,K)$ be a differentiable manifold with a geometric structure $K$. Then the following assertions are equivalent:
	\begin{enumerate}
		\item The manifold $(M,K)$ is a reductive locally homogeneous space, associated to the Lie pseudo-group $\G$.
		\item  There exists a connection $\TT$ such that:
	\begin{equation*}
		\TT \curv = 0, \quad \TT \tors = 0, \quad \TT K = 0,
	\end{equation*}
		where $\curv$ and $\tors$ are the curvature and torsion of $\TT$ respectively.
	\end{enumerate}
\end{theorem}

\begin{proof}
	Given a Lie pseudo-group $\mathcal{G}$ acting transitively on $(M,K)$ in a reductive locally fashion, let $(P \to M, \mathbf{H})$ the principal bundle associate to the structure of reductive locally homogeneous space as in \eqref{Equation 3}, for a fixed frame $u_0 \in \LM$.
	We define a horizontal distribution $D$ in $P$ by $D_u =  \Psi _* (\m)$, $\Psi = \psi _*$, for the unique $\psi \in \mathcal{G}$ such that $\psi _* (u_0) =u$, where $T_{u_0}P=\h + \m$ is the reductive decomposition. The distribution $D$ is also $\mathbf{H}$-invariant, that is, given $Y=\Psi_* (X)\in D_u$, $X \in \m$, we have that $(R_h)_* (Y) \in D_{u\cdot h} $, for $h\in \mathbf{H}$. Indeed, we write $X = \dt (\varphi_t)_* (u_0)$ and, by \eqref{Equation 2}, $h = u_0^{-1} \circ \varphi_* \circ u_0$ for certain $\varphi \in H(p_0)$. Then
	\begin{align*}
	(R_h)_* (Y) &= (R_h\circ \Psi)_* (X) = \dt R_h \circ \psi_* \circ (\varphi_t)_* (u_0) \\
	&= \dt  \psi_* \circ (\varphi_t)_* (u_0 \circ u_0^{-1} \circ \varphi_* \circ u_0) \\
	&= \dt  \psi_* \circ (\varphi_t)_* \circ \varphi_* \circ u_0 \\
	&= \dt  \psi_* \circ \varphi_* \circ \varphi^{-1}_* (\varphi_t)_* \circ \varphi_* \circ u_0 \\
	&= (\psi_* \circ \varphi_*)_* \Ad(H(p_0))_{\varphi^{-1}}(X).
	\end{align*}
	As $\Ad(H(p_0))_{\varphi^{-1}}(X) \in \m$ by reductive condition, and $\psi \circ \varphi \in \G$ we get the invariance. This means that $D$ can be understood as a linear connection $\tilde{\nabla}$.
	
	We now show that
	\begin{equation*}
	\TT \curv = 0, \quad \TT \tors = 0, \quad \TT K = 0.
	\end{equation*}
	For $p,q\in M$, let $\gamma$ be a path connecting them. The horizontal lift $\tilde{\gamma}$ with respect to $\tilde{\nabla}$ from $u\in P_p$ to $v\in P_q$, can be regarded as the parallel transportation $T_pM \to T_qM$.  But since $v = \psi _* u$, for an element $\psi \in \mathcal{G}$, we have that the parallel transportation is exactly $\psi _*$. We have that $\psi_*$ preserves $K$ and the connection $\TT$ (and hence, its curvature and torsion) by construction. Therefore, $K$, $\tilde{R}$ and $\tilde{T}$ are invariant under parallel transportation and their covariant derivatives vanish.
	
	Conversely, given a linear connection $\TT$ such that $\TT \curv = 0, \, \TT \tors = 0, \, \TT K = 0$, let $\mathcal{G}$ the its Lie pseudo-group of local transvections. Since $\TT K=0$, the elements of $\mathcal{G}$ preserve $K$. Furthermore, see \cite[V. I, p. 262, Cor. 7.5]{KN1963}, $\mathcal{G}$ acts transitively.
	
	To finish the proof we only have to show the reductive condition. Let $(\tilde{P}(u_0) \to M, \THol(u_0))$ be the holonomy reduction of the frame bundle associated to $\TT$ and an element $u_0 \in \LM$. We first prove that $\mathcal{G}(p_0)$ acts transitively on $\tilde{P}(u_0)$, being $p_0 = \pi(u_0)$. Given $v\in \tilde{P}(u_0) $, there exists a horizontal curve connecting $u_0$ with $v$. The projection to $M$ of that curve can be regarded as a parallel transportation from $p_0$ to $q = \pi(v)$ that, in addition, preserves curvature and torsion. Hence by \cite[V. I, p. 261, Thm. 7.4]{KN1963} there exists a local transvection $\psi \in \mathcal{G}(p_0)$ from $p_0$ to $q$ such that $\psi _*$ is that parallel transportation. Therefore, $\varphi _* (u_0) = v$ and $\mathcal{G}(p_0)$ acts transitively on $\tilde{P}(u_0)$. By construction, $P(u_0)$ (see \eqref{Equation 3}) coincides with $\tilde{P}(u_0)$. In particular, $\THol(u_0) = \mathbf{H}(u_0)$ which is closed and the effective condition it is satisfied because Proposition \ref{Proposition RL3}.
	
	Finally, if we consider $T_{u_0} \tilde{P}(u_0) = \h + \m$, where $\m$ is the horizontal distribution of $\TT$. To prove that $\m$ is $\mathrm{Ad}(H(p_0))$-invariant, let $X =  \dt (\varphi_t)_* (u_0) \in \m$, $\varphi \mathcal{G}(p_0)$ such that $\varphi(p_0)=p_0 $ and $h = u_0^{-1} \circ (\varphi^{-1})_* \circ u_0 \in \mathbf{H}(u_0)$. We consider,
	\begin{align*}
	\Ad(H(p_0))_{\varphi}(X) &= \dt \varphi_* \circ (\varphi_t)_* \circ \varphi^{-1}_* (u_0) \\
	&= \dt \varphi_* \circ (\varphi_t)_* \circ u_0 \circ u_0^{-1} \circ \varphi^{-1}_* u_0 \\
	&= \dt R_h \circ \varphi_* \circ  (\varphi _t)_*(u_0) = (R_h \circ \varphi_*)_* (X).
	\end{align*}
	Hence, $(R_h \circ \varphi_*)_* (X)$ belongs to the horizontal distribution ($\m$), because affine transvections preserve the horizontal distribution.
\end{proof}

If we apply this last Theorem in the framework of Theorem \ref{Theorem AG7} above, we get the following result.

\begin{corollary}\label{Corollary RL8}
	Let $(M, K)$ be a differentiable manifold with an affine connection $\T$. Then the following assertions are equivalent:
	\begin{enumerate}
		\item The manifold $(M,K)$ is a reductive locally homogeneous space, associated to a Lie pseudo-group contained in $\mathrm{Aff}_{loc}(M,\T)$.
		\item There exists a connection $\TT$ such that:
		\begin{equation*}
		\TT \curv = 0, \quad \TT \tors = 0,\quad \TT S = 0, \quad \TT K = 0,
		\end{equation*}
		or
		\begin{equation*}
		\TT R = 0, \quad \TT T = 0,\quad \TT S = 0, \quad \TT K = 0,
		\end{equation*}
		where $R,T$ and $\curv, \tors$ are the curvature and torsion tensor of $\T$ and $\TT$, respectively, and $S= \T - \TT$ is the  tensor.
\end{enumerate}
\end{corollary}

\begin{definition}\label{Definition RL9}
	Let $(M,K, \T)$ a manifold endowed with a geometrical structure and an affine connection $\T$. We will say that $(M,K, \T)$ is a \emph{reductive locally homogeneous manifold with $\T$} if it is reductive locally homogeneous associated to a Lie pseudo-group contained in $\mathrm{Aff}_{loc}(M,\T)$.
\end{definition}


\section{AS-manifolds and Homogeneous Structures} \label{Section AS}
 
In the previous section, we have proven that locally homogeneous and reductive manifolds are AS-manifolds, and vice versa. We now study AS-manifolds from an infinitesimal, or even pointwise, point of view.

Let $V$ be a vector space of dimension $n$. Let
\begin{equation*}
	\curv: V \wedge V \to \mathrm{End}(V), \quad \tors: V \to \mathrm{End}(V),
\end{equation*}
be linear homomorphisms and let $K$ be a set of linear tensors on $V$. We will say that $(\curv,\tors)$ is an  \textit{infinitesimal model associated to $K$} if it satisfies
\begin{align}
	&\tors_XY +\tors_YX = 0,							\label{Equation 4}\\
	&\curv_{XY}Z + \curv_{YX}Z = 0,					\label{Equation 5}\\
	&\curv_{XY} \cdot \tors = \curv_{XY} \cdot \curv = 0,\label{Equation 6}\\
	&\SC{XYZ} \curv_{XY}Z + \tors_{\tors_XY}Z = 0,	\label{Equation 7}\\
	&\SC{XYZ} \curv_{\tors_XY Z} = 0,				\label{Equation 8}\\
	&\curv_{XY}\cdot K = 0,					\label{Equation 9}
\end{align}
where $\SC{XYZ}$ is the cyclic sum, and $\curv_{XY}$ acts in a natural way in the tensor algebra of $V$ as a derivation. In addition, we say that two infinitesimal model $(V, \curv,\tors)$ and $(V', \curv',\tors')$ are \emph{isomorphic} if there exists a linear isomorphism $f: V \to V'$ such that
\begin{equation}							\label{Equation 10}
f\, \curv = \curv', \quad f\, \tors = \tors', \quad f\, K = K'.
\end{equation}
This notion of infinitesimal model is a generalization of the one given, for example, in \cite{N1954, TL1993}.

\begin{theorem} \label{Theorem AS1}
	Given a point $p_0\in M$ of an AS-manifold $(M,K, \TT)$, then $(V = T_{p_0}M, \tors_{p_0}, \curv_{p_0})$ is an infinitesimal model associated to $K_{p_0}$, where $\curv$ and $\tors$ are the curvature and torsion of $\TT$.
\end{theorem}

\begin{proof}
	Let $(M,K, \TT)$ be an AS-manifold. It satisfies
	\begin{equation*}
	\TT \curv = 0, \quad \TT \tors = 0, \quad \TT K = 0,
	\end{equation*}
	given a point $p_0\in M$ and we recall $V=T_{p_0}M$,  $\curv_0 = \curv_{p_0}$, $ \tors_0 = \tors_{p_0}$ and $K_0=K_{p_0}$, hence, $(\curv_0,\tors_0)$ is an infinitesimal model. Indeed,  we deduce \eqref{Equation 4} and \eqref{Equation 5} from the skew-symmetric definition of torsion and curvature. Equations \eqref{Equation 6} and \eqref{Equation 9} come from $\TT \curv = 0, \, \TT \tors = 0,\, \TT K = 0$. Finally, equations \eqref{Equation 7} and \eqref{Equation 8} are the Bianchi identities.
\end{proof}

Note that, Theorem \ref{Theorem AS1} provides an infinitesimal model for every point in an AS-manifolds. Now, we show it does not matter the chosen point $p_0\in M$. Indeed,
\begin{theorem}\label{Theorem AS2}
	Let $(M , K, \TT)$ be an AS-manifold. Given two different points $p_0$, $p_1 \in M$ their associated infinitesimal models are isomorphic.
\end{theorem}

\begin{proof}
	By \cite[Vol. 1, p. 262, Cor. 7.5]{KN1963}, there exists a locally affine transformation $\varphi$ sending $p_0$ to $p_1$. Because of being affine, we have that $\varphi_*$ is a linear isomorphism between $T_{p_0}M$ and $T_{p_1}M$ satisfying $\varphi_* \tors_{p_0} = \tors_{p_1}$ and $\varphi_* \curv_{p_0} = \curv_{p_1}$. By $\TT K = 0$, we conclude that $\varphi_* K_{p_0} = K_{p_1}$.
\end{proof}

Hence, associated to any AS-manifold there exists an unique infinitesimal model up to isomorphism. Furthermore, when different manifolds have isomorphic associated infinitesimal models, from \cite[Vol. 1, p.261, Thm. 7.4]{KN1963} we get the following result.

\begin{theorem}\label{Theorem AS3}
	Let $(M , K, \TT)$ and $(M', K', \TT')$ be two AS-manifold and let $p_0 \in M$ and $p_0' \in M'$ be two points, such that their associated infinitesimal models are isomorphic. Then there exists a local affine diffeomorphism between $p_0$ and $p'_0$ sending to $K_{p_0}$ to $K'_{p'_0}$.
\end{theorem}

So, we define the notion of AS-isomorphism between AS-manifolds.

\begin{definition} \label{Definition AS4}
		Let $(M , K, \TT)$ and $(M', K', \TT')$ be two AS-manifold and let $p_0 \in M$ and $p_0' \in M'$ be two point. We say $(M , K, \TT)$ and $(M', K', \TT')$ are AS-isomorphic if there exists a local affine diffeomorphism between $p_0$ and $p_0'$ sending $K$ to $K'$.
\end{definition}

From every infinitesimal model $(\curv, \tors)$ on $V$ associated to $K$, we can construct a transitive Lie algebra using the so-called \textit{Nomizu construction}, see \cite{N1954}. Let
\begin{equation}\label{Nomizu construction}
\g_0 = V \oplus \h_0,
\end{equation}
where $\h_0 = \{ A \in \mathfrak{end}(V): A \cdot \curv = 0,\, A \cdot \tors = 0,\, A \cdot K = 0 \}$, equipped with the Lie bracket
\begin{equation}\label{Nomizu corchetes}
	\begin{alignedat}{3}
		&[A,B] = AB -BA, &A,B \in \h_0, \\
		&[A,X] = AX, &A \in \h_0, \, X\in V,  \\
		&[X,Y] = -\tors_XY +\curv_{XY}, &X,Y \in V.
	\end{alignedat}
\end{equation}

Alternatively, we can also consider the so-called \textit{transvection algebra} $\g_0 '= V \oplus \h_0' $ \cite{K1990}, where $\h_0'$ is the lie algebra of endomorphism generated by $\curv_{XY}$ with $X,Y \in V$, equipped with brackets as above. In particular, this Lie algebra coincides with the holonomy algebra of the connection $\TT$. Then we have shown that every infinitesimal model has  Nomizu and transvection constructions.

Two Nomizu constructions $(\g_0, \h_0, \tors, \curv, K)$ and $(\g_0',\h_0', \tors', \curv', K')$ are isomorphic if there exists a Lie algebra isomorphism $F: \g_0 \to \g_0'$ such that $F (V)= V'$, $F$ sends $K$ to $K'$ and $F(\h_0) = \h_0'$.

\begin{proposition}\label{Proposition AS5}
	Two infinitesimal model are isomorphic if and only if their Nomizu constructions are isomorphic.
\end{proposition}
\begin{proof}
	Suppose that $V$ and $V'$ are two vector space with two infinitesimal models $(\curv,\tors)$ and $(\curv',\tors')$. Then there is an isomorphism $f: V \to V'$ such that satisfies \eqref{Equation 10}. We thus consider $\tilde{f}: \g_0 \to \g_0'$ such that $\tilde{f}|_V = f$ and $\tilde{f}|_{\h_0}(A)=  f \circ A  \circ f^{-1}$.
	
	Conversely, given a Lie algebra homomorphism $F: \g_0 \to \g_0'$ such that $F(V)= V'$ and $F(\h_0) = \h_0'$, then $f = F|_V$ is the isomorphism. Indeed, by definition $f$ sends $K$ to $K'$ and, taking into account that $F$ is a Lie algebra morphism, we obtain that $f$ sends $\curv$  to $\curv'$ and $\tors$ to $\tors'$.
\end{proof}

Surprisingly, the converse is no true: two different Nomizu constructions could give rise to the same Lie algebra, see \cite[p. 36]{L2014}.

Summarizing, we have proved that there exists a morphism from the class of AS-manifolds to the class of infinitesimal models. Moreover, every infinitesimal model has associated a Nomizu construction. Now, we prove the main theorem of the section, which shows that the morphism is surjective. Note that obviously it can not be injective. The proof of this result for Riemannian manifolds can be found in \cite{TL1993}.

\begin{theorem} \label{Modelo infinitesimal implica existencia de AS}
	Let $V$ be a vector space and $(\curv_0,\tors_0)$ an infinitesimal model associated to tensors $K_0$. Then, there is an AS-manifold $M$ with a geometrical structure defined by the tensor field $K$ and a point $p_0 \in M$ such that
	\begin{equation*}
		K_{p_0}= K_0,
	\end{equation*}
	and the curvature $\curv$ and torsion $\tors$ of the AS-connection $\TT$ verify that $\curv _{p_0} = \curv_0$ and $\tors _{p_0} = \tors_0$.

Any other manifold satisfying all this is locally affine diffeomorphic to $(M,\TT)$.
\end{theorem}

\begin{proof}
	From the model $(\curv_0, \tors_0)$ on $V$ associated to $K_0$, we can construct its transitive Lie algebra using the Nomizu construction, $\g_0 = V \oplus \h_0$ (see \eqref{Nomizu construction}).  We consider a basis $\{e_1,...,e_n\}$ of $V$ and a basis $\{A_1, ... , A_m\}$ of $\h_0$ and, respectively, we take its dual basis $\{\theta^1, ... , \theta^n\}$ and $\{ \omega^1, ..., \omega^m\}$. Let $G$ be the connected, simply connected Lie group associated to $\g_0$. Indeed, the vector elements $e_i$ and $A_k$ and the 1-forms $\theta^i$ and $\omega^k$ of $\g_0$ can be regarded as left-invariant vector fields and 1-forms on $G$.
	
	Let $\phi = (x^1, ... , x^{n+m}): U \subset G \to V \subset \R^{n+m}$ a local coordinate system defined on a neighbourhood $U$ of the identity $e$ of $G$ such that,
	\begin{equation}\label{diferenciales de la carta}
		dx^i |_e = \theta^i|_e, \quad ,i\in\{1,...,n\}.
	\end{equation}
	We now consider the immersion map $f: W \subset V \to G$ given by $f(y^1, ..., y^n) = \phi^{-1}(y^1, ...,y^n, 0, ..., 0)$ where $W$ is an open neighbourhood of  $0\in \R^n$. We define,
	\begin{equation*}
		\tilde{\theta}^i = f^*(\theta^i), \quad \tilde{e}_i = f^*(e_i), \quad i=1,...,n.
	\end{equation*}
	Because of \eqref{diferenciales de la carta}, the 1-forms $\tilde{\theta}^1,..., \tilde{\theta}^n$ are linearly independent at $0$. Then, let $M \subset W$ be the open neighbourhood of $0$ such that $\tilde{\theta}^1,..., \tilde{\theta}^n$ are linearly independent at each point of $M$.
	
	We consider $\omega_j^i = \sum_{k=1}^{m} \theta^i(A_k (e_j)) \omega^k$ and $\tilde{\omega}_j^i $ its pull back to $M$. Let $\TT$ be the linear connection whose connection form is $\tilde{\omega} = (\tilde{\omega}_j^i)$. We now consider the pull-back to $M$ of the extension $\curv$, $\tors$ and $K$ of the tensor elements $\curv_0$, $\tors_0$ and $K_0$ to the Lie group $G$. 	
	Note that $\tilde{\omega}$ takes values in $\h_0$. Therefore, by definition of $\h_0$, we have $\omega(X) \cdot \curv = 0$, $\omega(X) \cdot \tors = 0$ and $\omega(X) \cdot K = 0$. Indeed, this last identities mean that $\TT$ makes parallel $\curv$, $\tors$ and $K$.
	
	To end, we need to prove that $\curv$ and $\tors$ are the curvature and torsion tensors of $\TT$. First, applying the definition of the exterior differential and \eqref{Nomizu corchetes}, we have that,
	\begin{align*}
		d \theta^i + \omega_j^i \wedge \theta^j &= \frac{1}{2} \sum_{j,k=1}^n \theta^i((\tors_0)(e_j,e_h)) \theta^j \wedge \theta^h\\
		d \omega_j^i + \sum_{k=1}^{n} \omega_k^i \wedge \omega_j^k &= \frac{1}{2} \sum_{h,k=1}^n \theta^k((\curv_0)(e_i,e_j,e_h)) \theta^h \wedge \theta^k.
	\end{align*}
	If we pull-back these two equations to $M$, they are the two structural equations for the connection form $\omega$. Therefore, the curvature and torsion of $\TT$ are the tensors whose components in the basis $\{\theta^1, ... , \theta^n\}$ are constants $\tilde{\theta}^i((\tors_0)(\tilde{e}_j,\tilde{e}_h))$ and
	$\theta^k((\curv_0)(\tilde{e}_i,\tilde{e}_j,\tilde{e}_h))$, respectively. Finally, the curvature and torsion of $\TT$ coincide with $\curv$ and $\tors$.
\end{proof}

We finally consider the particular case where $M$ is a manifold with a geometric structure $K$ equipped with a connection $\T$ and an AS-connection $\TT$ such that,
\begin{equation*}
\TT \curv = 0, \quad \TT \tors = 0,\quad \TT S = 0,\quad \TT K = 0,
\end{equation*}
where $S= \T- \TT$. So we can consider,
\begin{equation*}
T_{X}Y = \tilde{T}_{X}Y + S_{X}Y -S_{Y}X, \quad  R_{XY} = \tilde{R}_{XY} + [S_Y, S_X] - S_{S_{Y}X-S_{X}Y}.
\end{equation*}
where $R$ and $T$ are the curvature and torsion of $\T$.

\begin{corollary}\label{Corollary AS7}
	Let $(M, S, K, \TT)$ and $(M', S', K', \TT')$ be two AS-manifolds with homogeneous structures $S$ and $S'$. Then, there exists an AS-isomorphism between $M$ and $M'$ if and only if there exists an affine local diffeomorphism between $(M,\T)$ and $(M', \T')$ sending $S$ to $S'$ and $K$ to $K'$.
\end{corollary}

Given a fixed point $p_0\in M$, by Theorem \ref{Theorem AS1}, we can consider an infinitesimal model $(T_{p_0}M,\curv_{p_0}, \tors_{p_0})$ associated to $S_{p_0}$ and $K_{p_0}$ with
\begin{align*}
(T_{p_0})_{X}Y &= (\tilde{T}_{p_0})_{X}Y + (S_{p_0})_{X}Y -(S_{p_0})_{Y}X,\\
(R_{p_0})_{XY} &= (\tilde{R}_{p_0})_{XY} + [(S_{p_0})_Y, (S_{p_0})_X] + (S_{p_0})_{(S_{p_0})_{Y}X-(S_{p_0})_{X}Y},
\end{align*}
where $R_{p_0}$ and $T_{p_0}$ are the curvature and torsion of $\T$ in $p_0$.

\begin{corollary}\label{Corollary AS8}
	Let $(V, \curv, \tors)$ and $(V', \curv', \tors ')$ be two infinitesimal models associated to $S$, $K$ and $S'$, $K'$, respectively, with
	\begin{align*}
	T_{X}Y &= \tilde{T}_{X}Y + S_{X}Y -S_{Y}X, \quad  R_{XY} = \tilde{R}_{XY} + [S_Y, S_X] - S_{S_{Y}X-S_{X}Y},\\
	T'_{X}Y &= \tors'_{X}Y + S'_{X}Y -S'_{Y}X, \quad  R'_{XY} = \tilde{R}'_{XY} + [S'_Y, S'_X] - S'_{S'_{Y}X-S'_{X}Y}.
	\end{align*}
	Hence, there exists an isomorphism of infinitesimal models if and only if there exists a linear isomorphism $f: V \to V'$ such that,
	\begin{equation} \label{Equation 14}
	 f\, R = R', \quad f\, T = T' \quad f\, S = S', \quad f\, K=K'.
	\end{equation}
\end{corollary}


\section{\Large{Invariant $\mathrm{Sp}(V, \omega)$-submodules of $S^2 V^* \otimes V^*$ and $\wedge^2 V^* \otimes V^*$}}\label{Section IS}

Let $(V, \omega)$  be a symplectic vector space. Based on the classifications given in \cite{AP2011}, we give below explicit expressions of the invariant $\mathrm{Sp}(V)$-submodules of $S^2 V^* \otimes V^*$ and $\wedge^2 V^* \otimes V^*$. For that, we identify a symplectic vector space $V$ and its dual $V^*$ as
\begin{align*}
(\cdot)^*: V &\to V^* \\
X &\longmapsto X^*(Y) = \omega(X,Y).
\end{align*}
Furthermore, we can transfer the symplectic form to $V^*$ as $\omega^*(X^*,Y^*) = \omega(X,Y)$, that is, we regard $(V, \omega)$ and $(V^*, \omega^*)$ as symplectomorphic.

For the sake of simplicity, from now on, we denote $\omega_{XY} = \omega(X,Y)$.

\begin{theorem}\label{Theorem IS1}
	If $n \geq 2$, the space of cotorsion-like tensors has the decomposition in irreducible $\mathrm{Sp}(V)$-submodules as
	\begin{equation*}
	S^2 V^* \otimes V^* = \mathcal{S}_1 (V) + \mathcal{S}_2 (V) + \mathcal{S}_3 (V)
	\end{equation*}
	where,
	\begin{align*}
	\mathcal{S}_1(V) &= \{ S \in S^2 V^* \otimes V^*:
	\: S_{XYZ} = \omega_{ZY}\omega_{XU}+ \omega_{ZX}\omega_{YU}, U \in V \},
	\\
    \mathcal{S}_2(V) &= \{ S \in S^2V^* \otimes V^*:\:
	\SC{XYZ}  S_{XYZ} = 0,\, \mathrm{s}_{13}(S)=0\},
	\\		
	\mathcal{S}_3(V) &=  \{ S \in S^2V^* \otimes V^*:\: S_{XYZ}= S_{XZY} \} = S^3 V^*,
	\end{align*}
	and
    \[
    \mathrm{s}_{13}(S)(Z)=\sum _{i=1}^n \left(S_{e_i Z e_{i+n}}-S_{e_{i+n} Z e_{i}}\right),
    \]
    for a symplectic base $\{ e_1,\ldots,e_n,e_{n+1},\ldots, e_{2n}\}$. If $n =1$, then $S^2 V^* \otimes V^* = \mathcal{S}_1 (V) + \mathcal{S}_3 (V)$.

    The dimensions of the subspaces are
	\begin{equation*}
	\dim (\mathcal{S}_1(V)) =2n, \quad
	\dim (\mathcal{S}_2(V)) =\frac{8}{3}(n^3-n), \quad
	\dim (\mathcal{S}_3(V)) =\small{\begin{pmatrix}
		2n+2\\
		3
		\end{pmatrix}}.
	\end{equation*}

\end{theorem}

\begin{proof}
	Given a symplectic basis $\{ e_1, \ldots ,e_n, e_{n+1}, \ldots, e_{2n}\}$ of $V$, we define the morphisms
	\begin{align*}
	\varphi : S^2 V^* \otimes V^* &\to V^* \\
	(u_1^* u_2 ^* \otimes v^*) & \longmapsto \omega_{u_1,v} u_2^* + \omega_{u_2, v} u_1^*,\\
	\pi : S^2 V^* \otimes V^* &\to S^3 V^* \\
	(u_1^* u_2 ^* \otimes u^*) & \longmapsto \frac{1}{3} u^* u_1^* u_2^*,
	\end{align*}
	and
	\begin{align*}
	\xi :  V^* &\to S^2 V^* \otimes V^* \\
	u^* & \longmapsto \frac{1}{2n+1} \sum_{i=1}^{n} e_i^* u^* \otimes e_{i+n}^* - e_{i+n}^* u^* \otimes e_{i}^*.
	\end{align*}
	By Theorem 1.1 in \cite{AP2011}, applied to $(V^*,\omega^*)$, we decompose
	\begin{equation*}
	S^2 V^* \otimes V^* = S^3 V^* + \mathcal{A}' + V^*
	\end{equation*}
	where $\mathcal{A}' = \ker(\varphi)\cap \ker(\pi)$ and $V^*$ is isomorphic to $\mathrm{im}(\xi)$. We define $\mathcal{S}_1(V):= V^*$, $\mathcal{S}_2(V):= \mathcal{A}'$ and $\mathcal{S}_3 (V) := S^3 V^*$.
	
	For the explicit expression of $\mathcal{S}_1(V)$, given $W^* \in V^*$ we have
	\begin{align*}
	\xi(W^*)_{XYZ} &= \frac{1}{2n+1}( \sum_{i=1}^{n}\, x_{i+n} \omega_{WY} z_i + y_{i+n}\omega_{WX} z_i \\
	& \quad -  x_i\omega_{WY} z_{i+n} - y_i \omega_{WX} z_{i+n})\\
    &= \frac{1}{2n +1} ( \omega_{ZX} \omega_{WY} + \omega_{ZY} \omega_{WX}).
    \end{align*}
    Hence, taking $U = \frac{1}{2n +1} W $, we get the required result for $\mathcal{S}_1(V)$.
	
	With respect to the explicit expressions of $\mathcal{S}_2(V)$, for
	\[
	S =\frac{1}{2} \sum_{i,j,k=1}^{2n} S_{e_i e_j e_k} e_i^*e_j^*\otimes e_k^* \in S^2 V^* \otimes V^*,
	\]
	we have
	\begin{align*}
		\varphi(S) 	
		&= \frac{1}{2} \sum_{i,j,k=1}^{2n} S_{e_i e_j e_k} \mathrm{s}_{13}(e_i^*e_j^*\otimes e_k^*)  \\
		&= \frac{1}{2} \sum_{i,j,k=1}^{2n} S_{e_i e_j e_k} ( \omega_{e_ie_k} e_j^* + \omega_{e_je_k}  e_i^*)   \\
		&= \sum_{i,j,k=1}^{2n} \frac{1}{2} (S_{e_i e_j e_k}+ S_{e_j e_i e_k})\omega_{e_ie_k}e_j^* \\
		&= \sum_{j=1}^{2n} \sum_{i=1}^{n} \frac{1}{2} ( S_{e_i e_j e_{i+n}} + S_{e_j e_i e_{i+n}} - S_{e_{i+n} e_j e_i} - S_{e_j e_{i+n} e_i})e_j^* \\
		&= \sum_{j=1}^{2n} \sum_{i=1}^{n} (S_{e_i e_j e_{i+n}}- S_{e_{i+n}e_j e_i})e_j^*.
	\end{align*}
	Hence, $S\in \mathrm{ker}\varphi$ if and only if $\mathrm{s}_{13}(S)=0$ as in the statement. Moreover, $ \frac{1}{3} e_i^*e_j^*e_k^* = e_i^*e_j^*\otimes e_k^* + e_k^*e_i^*\otimes e_j^* +e_j^*e_k^*\otimes e_i^*$ and therefore
	\begin{equation*}
		\pi(S)_{XYZ}	=  \SC{XYZ} S_{XYZ},
	\end{equation*}	
so that we have the expression for the tensors in $\mathcal{S}_2(V)$. 	
	The dimensions come from Theorem 1.1 of \cite{AP2011}.
\end{proof}

Now, using these expressions we are going to give the explicit classes of torsion-like tensors.

\begin{theorem} \label{Theorem IS2}
	If $n \geq 3$, the space of torsion-like tensors has the decomposition in irreducible $\mathrm{Sp}(V)$-submodules as
	\begin{equation*}
	\wedge^2 V^* \otimes V^* =
	\tilde{\mathcal{T}}_1(V)+\tilde{\mathcal{T}}_2(V)+\tilde{\mathcal{T}}_3(V)+\tilde{\mathcal{T}}_4(V)
	\end{equation*}
	where
	\begin{align*}
	\tilde{\mathcal{T}}_1(V) = \{\tors \in  \wedge^2 V^* \otimes V^*&:\: \tors_{XYZ}= 2 \omega_{XY}\omega_{ZU} + \omega_{XZ}\omega_{YU} - \omega_{YZ} \omega_{XU}, \, U \in V\}, \\
	\tilde{\mathcal{T}}_2(V) = \{ \tors \in \wedge^2 V^* \otimes V^* &:\: \SC{XYZ} \tors_{XYZ} = 0, \, \mathrm{t}_{12}(\tors)= 0\},\\
	\tilde{\mathcal{T}}_3(V) = \{\tors \in  \wedge^2 V^* \otimes V^* &:\: \tors_{XYZ} =\omega_{XY}\omega_{UZ} + \omega_{YZ} \omega_{UX} + \omega_{ZX}\omega_{UY}, \, U \in V\},\\
	\tilde{\mathcal{T}}_4(V) = \{\tors\in \wedge^2 V^* \otimes V^*&:\: \tors_{XYZ} = - \tors_{XZY}, \, \mathrm{t}_{12} (\tors) = 0\},
	\end{align*}
	and
	\begin{align*}
		\mathrm{t}_{12}(\tors)(Z)&=\sum _{i=1}^n \left(\tors_{e_i e_{i+n} Z}\right),
	\end{align*}
	for a symplectic basis $\{ e_1,\ldots,e_n,e_{n+1},\ldots, e_{2n}\}$. If $n=2$, then $\wedge^2 V^* \otimes V^* =
	\tilde{\mathcal{T}}_1(V)+\tilde{\mathcal{T}}_2(V)+\tilde{\mathcal{T}}_4(V)$. If $n = 1$, then $\wedge^2 V^* \otimes V^* =
	\tilde{\mathcal{T}}_1(V)$.
	
	In addition,
	\small{
	\begin{equation*}
	\mathrm{dim} (\tilde{\mathcal{T}}_1(V)) = \mathrm{dim} (\tilde{\mathcal{T}}_3(V)) = 2n,  \quad
	\mathrm{dim} (\tilde{\mathcal{T}}_2(V)) = \frac{8}{3}(n^3-n), \quad \mathrm{dim} (\tilde{\mathcal{T}}_4(V)) = \frac{2}{3}n(2n^2 - 3n-2).
	\end{equation*}}
\end{theorem}

\begin{proof}
	For a symplectic basis $\{ e_1, \ldots ,e_n, e_{n+1}, \ldots, e_{2n}\}$ of $V$, we define the morphisms
	\begin{align*}
	A_2: S^2 V^* \otimes V^*& \to  \: \wedge  ^2 V^* \otimes V^* \\
	(u_1^*u_2^* \otimes v^*)& \longmapsto  \: v^* \wedge u_1^* \otimes u_2^* + v^* \wedge u_2^*\otimes u_1 ^*, \\
	C: \wedge^2 V^* \otimes V^* & \to V^* \\
	(u_1^*\wedge u_2^* \otimes v^*) & \longmapsto \omega_{u_1 u_2}v^* + \omega_{vu_1}u_2^* + \omega_{u_2 v}u_1^*,
	\end{align*}
	and
	\begin{align*}
	\eta: V^* &\to \wedge^3 V^* \\
	u^* &\longmapsto \sum_{i=1}^{n} e_i^* \wedge e_{i+n}^* \wedge u^*.
	\end{align*}
	By Theorem 1.2 of \cite{AP2011}, applied to $(V^*,\omega^*)$, we decompose
	\begin{equation*}
	\wedge^2 V^* \otimes V^* = V_1^*+\mathcal{A}'+V_2^*+\mathcal{T}'
	\end{equation*}
	where, $V_1^* = A_2(\mathcal{S}_1(V))$, $\mathcal{A}' = A_2(\mathcal{S}_2(V))$, $\mathcal{T}' = \ker C \cap \wedge^3 V^*$ and $V_2^*= \mathrm{Im}(\eta)$ is the vector space such that $V_2^* \subset \wedge^3 V^*$ and $V_2^* + \mathcal{T}' = \wedge^3 V^*$. We define $\tilde{\mathcal{T}}_1(V):= V_1^*$, $\tilde{\mathcal{T}}_2(V):= \mathcal{A}'$, $\tilde{\mathcal{T}}_3 (V) := V_2^*$ and $\tilde{\mathcal{T}}_4(V) : = \mathcal{T}'$.
	
	First, as
	\begin{equation} \label{Equation 15}
		A_2(S)_{XYZ} = S_{YZX}- S_{XZY},
	\end{equation}
we get the expression for the tensors in $\tilde{\mathcal{T}}_1(V)$ in view of the expression of $\mathcal{S}_1(V)$ in Theorem \ref{Theorem IS1}.

	Indeed, by equation \eqref{Equation 15}, we infer the explicit expression of $\tilde{\mathcal{T}}_1(V)$.
	
	To study the explicit expression of $\mathcal{T}_2(V)$, we have to consider the following exact sequence, \cite[Eq. (1.3)]{AP2011},
	\begin{equation*}
		0 \to S^3 V^* \xrightarrow{A_1} S^2V^* \otimes V^* \xrightarrow{A_2} \wedge^2 V^* \otimes V^* \xrightarrow{A_3} \wedge^3 V^* \to 0
	\end{equation*}
	where $A_1 = \pi$ and $A_3(u_1^* \wedge u_2^* \otimes v^*) = u_1^* \wedge u_2 ^*  \wedge v^*$.
	Note that, $ u_1^* \wedge u_2 ^*  \wedge v^* = u_1^* \wedge u_2 ^*  \otimes v^* +  v^* \wedge u_1^*  \otimes u_2^* + u_2^* \wedge  v^*  \otimes u_1^*$, hence,
	\begin{equation*}
		A_3(\tors)_{XYZ} = \SC{XYZ} \tors_{XYZ}.
	\end{equation*}
	
	Therefore, $\tilde{\mathcal{T}}_2(V)$ is generated by $\tors_{XYZ} = S_{YZX}- S_{XZY}$ with $S \in S^2 V^* \otimes V^*$ and $\mathrm{s}_{13}(S)= 0$. The first condition is equivalent to $\tors \in \mathrm{ker}(A_3)$, or equivalently, $ \SC{XYZ} \tors_{XYZ} = 0$. The second condition is equivalent $\mathrm{t}_{12} (T) = 0$ straightforwardly.
	For the explicit expressions of $ \tilde{\mathcal{T}}_4(V)$, given
	\begin{equation*}
		\tors = \frac{1}{2} \sum_{i,j,k=1}^{2n} \tilde{T}_{e_i e_j e_k} e_i^*\wedge e_j^* \otimes e_k^* \in \wedge^2 V^* \otimes V^*,
	\end{equation*}
	we have
	\begin{align*}
		C (\tors)
		&= \frac{1}{2}  \sum_{i,j,k=1}^{2n} \tors_{e_i e_j e_k} C(e_i^*\wedge e_j^* \otimes e_k^*)  \\
		&= \frac{1}{2}  \sum_{i,j,k=1}^{2n} \tors_{e_i e_j e_k} \left(\omega_{e_i e_j} e_k^* + \omega_{e_k e_i} e_j^* + \omega_{e_j e_k} e_i^*\right) \\
		&= \frac{1}{2}\left(  \sum_{i,j,k=1}^{2n} \tors_{e_i e_j e_k} \omega_{e_i e_j} e_k^* +  \sum_{i,j,k=1}^{2n} \tors_{e_i e_j e_k} \omega_{e_k e_i} e_j^* +   \sum_{i,j,k=1}^{2n} \tors_{e_i e_j e_k} \omega_{e_j e_k} e_i^*\right) \\
		&= \frac{1}{2}  \sum_{i,j,k=1}^{2n} \left(\tors_{e_i e_j e_k} + \tors_{e_k e_i e_j}+\tors_{e_j e_k e_i}\right) \omega_{e_i e_j} e_k^* \\
		&=   \sum_{k=1}^{2n} \sum_{i=1}^{n} \frac{1}{2}\left(\tors_{e_i e_{i+n} e_k} + \tors_{e_k e_i e_{i+n}}+ \tors_{e_{i+n} e_k e_i}- \tors_{e_{i+n} e_i e_k} - \tilde{T}_{e_k e_{i+n} e_i} - \tilde{T}_{e_i e_k e_{i+n}} \right) e_k^* \\
		&= \sum_{k=1}^{2n} \sum_{i=1}^{n} \left(\tors_{e_i e_{i+n} e_k} + \tors_{e_k e_i e_{i+n}}+ \tors_{e_{i+n} e_k e_i} \right) e_k^*.
	\end{align*}
	Therefore, for $\tilde{T}\in \wedge ^3 V^*$, $C(\tors)=0$ is equivalent to $\mathrm{t}_{12} (\tors) = 0$.
	
	Finally, with respect to the explicit expressions of $\tilde{\mathcal{T}}_3(V)$, given $U^* \in V^*$ with dual element $U  \in V$,
	\begin{align*}
	\eta (U^*)
	&= \sum_{i=1}^{n} e_i^* \wedge e_{i+n}^* \wedge U^* \\
	&= \sum_{i=1}^{n} \left(e_i^* \wedge e_{i+n}^* \otimes U^* + U^* \wedge e_i^* \otimes e_{i+n}^* + e_{i+n}^* \wedge U^* \otimes e_i^*\right),
	\end{align*}
	evaluating in $X,Y,Z$, we infer,
	\begin{align*}
	\eta(U^*)_{XYZ} &=  \sum_{i=1}^{n}\hspace{0.4cm}  \left((x_i y_{i+n}- x_{i+n} y_i) \omega_{UZ} + \right.\\
	&\hspace{1cm}+ (\omega_{UX} y_{i+n} - \omega_{UY} x_{i+n}) (-z_i) + \\
	&\hspace{1cm}\left. +  (- x_i \omega_{UY} - (-y_i) \omega_{UX}) z_{i+n}\right) \\
	&= \omega_{XY}\omega_{UZ} + \omega_{YZ} \omega_{UX} + \omega_{ZX}\omega_{UY}
	\end{align*}
	Therefore, $\tilde{\mathcal{T}}_3(V)$ has the claimed form.
	
\end{proof}

\begin{remark}\label{Remark IS3}
We have the following sums
\begin{itemize}
	\item $\tilde{\mathcal{T}}_1(V) + \tilde{\mathcal{T}}_2(V) = \{\tors \in \wedge^2 V^* \otimes V^* :\: \SC{XYZ} T_{XYZ} =  0\}$,
	\item $\tilde{\mathcal{T}}_2(V)+ \tilde{\mathcal{T}}_4(V) + W = \{\tors \in \wedge^2 V^* \otimes V^* :\: \mathrm{t}_{12} (\tors) = 0\}, $
    \item $\tilde{\mathcal{T}}_3(V) + \tilde{\mathcal{T}}_4(V) = \wedge^3 V^*$,
    \item $\tilde{\mathcal{T}}_2(V)+ \tilde{\mathcal{T}}_4(V) = \{\tors \in \wedge^2 V^* \otimes V^* :\: \mathrm{t}_{12} (\tors) = 0, \mathrm{t}_{13} (\tilde{T}) = 0\} $.
\end{itemize}
Where,
$$
W = \{\tilde{T} \in \wedge^2 V^* \otimes V^* :  \tilde{T}_{XYZ} = \omega_{XY} \omega_{ZU} - n \omega_{XZ} \omega_{YU} + n \omega_{YZ} \omega_{XU}, U \in V\}
$$
and
$$
\mathrm{t}_{13}(\tilde{T}) (Y) = \sum _{i=1}^n \left(\tors_{e_i Y e_{i+n}} - \tors_{e_{i+n} Y e_{i}}\right)
$$
The first two come directly from the expressions of the classes in the previous Theorem and the fact that $W$ is the linear subspace of $\tilde{\mathcal{T}}_2(V)+ \tilde{\mathcal{T}}_3(V)$ whose elements vanish for $\mathrm{t}_{12}$. With respect to the third identity, we note that $\mathrm{Id}_{V^*} = \frac{1}{3(n-1)} C \circ \eta$, so that we can decompose $\wedge^3 V^* = \ker C +  \mathrm{Im} \,\eta = \tilde{\mathcal{T}}_3(V) + \tilde{\mathcal{T}}_4(V)$. The last identity is consequence of $\mathrm{t}_{13}$ vanishes for $\tilde{\mathcal{T}}_2(V)+ \tilde{\mathcal{T}}_4(V)$ and does not vanish for $W$.
\end{remark}


\section{Classifications for almost symplectic and Fedosov AS-manifolds}\label{Section C}

\subsection{Almost symplectic AS-manifolds}

We now want to classify the infinitesimal models in the case of vector spaces $V$ endowed with a linear symplectic tensor $K=\omega$. If $(V, \curv, \tors)$ and $(V', \curv', \tors')$ are two infinitesimal models associated to symplectic linear tensors $\omega$ and $\omega'$, respectively, with $\mathrm{dim}V=\mathrm{dim}V'$, since there are symplectomorphisms between $V$ and $V'$, we can identify $V'$ with $V$ and $\omega'$ with $\omega$. From \eqref{Equation 10}, isomorphisms $f:V\to V$ of almost symplectic infinitesimal models satisfy
\begin{equation*}
f\, \curv= \curv', \quad f\, \tors = \tors', \quad f\, \omega = \omega,
\end{equation*}
and in particular $f \in \mathrm{Sp}(V,\omega)=\mathrm{Sp}(V)$. If we decompose curvature-like or torsion-like tensor spaces in $\mathrm{Sp}(V)$-irreducible submodules, then we get a necessary condition to be isomorphic as models, by virtue of Theorem \ref{Theorem AS3}, also as AS-manifolds.

For the classification of the torsion $\tors$ into $\mathrm{Sp}(V)$-classes, we will work both with $(1,2)$-tensors and $(0,3)$-tensors given by the isomorphism
\begin{equation*}
\tors_{XYZ}= \omega(\tors_XY, Z), \quad X,Y,Z \in V.
\end{equation*}

Let $(M, \omega)$ an almost symplectic AS-manifold. We denote by $\tilde{\mathcal{T}}$ the set of \emph{homogeneous almost symplectic torsions}, that is, the torsions of an AS-connection on $(M, \omega)$. Given any $p_0 \in M$, from Theorem \ref{Theorem AS1}, $(V=T_{p_0}M,\tilde{R}_{p_0},\tors _{p_0})$ is an infinitesimal model associated to $\omega _{p_0}$. Thus $T_{p_0} \in \wedge^2 V \otimes V$, and the classification given in Theorem \ref{Theorem IS2} gives us the following decomposition
\begin{equation*}
\tilde{\mathcal{T}} =
\tilde{\mathcal{T}}_1+\tilde{\mathcal{T}}_2+\tilde{\mathcal{T}}_3+\tilde{\mathcal{T}}_4
\end{equation*}
where
\begin{align*}
\tilde{\mathcal{T}}_1 &= \{\tors \in \tilde{\mathcal{T}} :\: \tors_{XYZ}= 2 \omega_{XY}\omega_{ZU} + \omega_{XZ}\omega_{YU} - \omega_{YZ} \omega_{XU}, \, U \in \mathfrak{X}(M) \}, \\
\tilde{\mathcal{T}}_2 &= \{ \tors \in \tilde{\mathcal{T}}:\: \SC{XYZ} \tors_{XYZ} = 0, \, \mathrm{t}_{12}(\tors)= 0\}\\
\tilde{\mathcal{T}}_3 &= \{\tors \in \tilde{\mathcal{T}}:\: \tors_{XYZ} =\omega_{XY}\omega_{UZ} + \omega_{YZ} \omega_{UX} + \omega_{ZX}\omega_{UY}, \, U \in \mathfrak{X}(M)\},\\
\tilde{\mathcal{T}}_4 &= \{\tors \in \tilde{\mathcal{T}} :\: \tors_{XYZ}=- \tors_{XZY},\, \mathrm{t}_{12} (\tors) = 0\}.
\end{align*}

\begin{definition}\label{Definition C1}
	Let $\tors \in \tilde{\mathcal{T}}$ be a homogeneous almost symplectic torsion. It is of \emph{type $i$}  if $\tors$ lies in $\tilde{\mathcal{T}}_i$ and correspondingly is of \emph{type $i+j$} if lies in $\tilde{\mathcal{T}}_i + \tilde{\mathcal{T}}_j$ with $i$, $j\in \{1,2,3,4\}$ and $i\neq j$.
\end{definition}
Summarizing, we described almost symplectic AS-manifolds in sixteen classes defined by its torsion tensor.

\begin{theorem} \label{Theorem C2}
	Let $(M, \omega)$ be an almost symplectic AS-manifold. Then, $(M, \omega)$ is a symplectic manifold if and only if the torsion of $\TT$ lies in $\tilde{\mathcal{T}_1} + \tilde{\mathcal{T}_2}$.
\end{theorem}
\begin{proof}
	If $(M, \omega)$ is a symplectic manifold, there is a torsion-free symplectic connection $\T$ (see \cite[Theorem 2.1]{AP2011}). The difference $S = \T - \TT$ is a  $(1,2)$-tensor such that $\tors_X Y = S_Y X - S_XY$. Then $\tors_{XYZ} =  A_2 (-S) = S_{XZY}- S_{YZX}$, where $S_{XYZ} = \omega(S_ZX, Y)$ and $\tors_{XYZ} =  \omega(\tors_XY,Z)$. In particular, $\tors$ lies in $\mathcal{T}_1 + \mathcal{T}_2$.
	
	Conversely, if $\tors$ lies in $\mathcal{T}_1 + \mathcal{T}_2$, then there exists at least one tensor $S \in S^2 T^*M \otimes T^*M$, such that, $\tors_{XYZ} = S_{YZX} -S_{XZY}$. We can consider the  tensor $S_XY$ with $ \omega(S_Z X, Y) = S_{XYZ}$. It satisfies that $\tors_X Y = S_XY - S_Y X$ with $\omega(\tors_XY,Z) = \tors_{XYZ}$ and preserves the symplectic form. The connection $\T = \TT - S$ is symplectic.
\end{proof}

\subsection{Fedosov AS-manifolds}

We now want to study infinitesimal models associated to a linear symplectic tensor $\omega$ and a homogeneous structure $S$ as in Corollary \ref{Corollary AS8}. Let $(V, \curv, \tors)$ and $(V', \curv', \tors ')$ be two infinitesimal models associated to (1,2) linear tensors $S$ and $S'$, respectively, with
\begin{align*}
T_{X}Y &= \tilde{T}_{X}Y + S_{X}Y -S_{Y}X, \quad  R_{XY} = \tilde{R}_{XY} + [S_Y, S_X] - S_{S_{Y}X-S_{X}Y},\\
T'_{X}Y &= \tors'_{X}Y + S'_{X}Y -S'_{Y}X, \quad  R'_{XY} = \tilde{R}'_{XY} + [S'_Y, S'_X] - S'_{S'_{Y}X-S'_{X}Y},
\end{align*}
and also associated to symplectic linear tensors $\omega$ and $\omega'$, respectively, with $\mathrm{dim}V=\mathrm{dim}V'$. Since there are symplectomorphisms between $V$ and $V'$, we can identify $V$ with $V'$ and $\omega$ with $\omega'$. Therefore, by \eqref{Equation 14}, there is a linear isomorphism $f: V \to V$ such that,
\begin{equation} \label{Equation 16}
f\, R= R', \quad f\, T = T', \quad f\, S = S', \quad f\, \omega = \omega,
\end{equation}
and in particular $f \in \mathrm{Sp}(V,\omega)=\mathrm{Sp}(V)$. If we decompose cotorsion-like, curvature-like or torsion-like tensor spaces in $\mathrm{Sp}(V)$-irreducible submodules, then we get a necessary condition to be isomorphic as models, by virtue of Theorem \ref{Theorem AS3}, also as AS-manifolds.

Let $(M,\omega, \T)$ be a Fedosov manifold ($\mathrm{dim}M=2n$), that is, a symplectic manifold equipped with affine and torsion free connection such that $\nabla \omega =0$ (cf. \cite{GRS1998}). Let $S = \T -\TT$ be a homogeneous structure tensor, i.e., 
\begin{equation*}
\TT R = 0, \quad \TT S = 0, \quad \TT \omega = 0.
\end{equation*}
Since $\nabla \omega =0$, the second condition is equivalent to $S \cdot \omega = 0$. We will work with $S$ both as a $(1,2)$-tensor and a $(0,3)$-tensor by the isomorphism
\begin{equation*}
S_{XYZ} = \omega(S_Z X, Y).
\end{equation*}
The condition $ S \cdot \omega = 0$ is equivalent to
\begin{equation*}
S_{XYZ} = S_{YXZ}
\end{equation*}
that is, $S\in S^2 T^*M \otimes T^*M$.

From Theorem \ref{Theorem AS1}, given $p_0\in M$, we can consider the infinitesimal model $(V = T_{p_0}M,\curv_{p_0}, \tors_{p_0})$ associated to $S_{p_0}$ and $\omega_{p_0}$ with
\begin{align*}
(T_{p_0})_{X}Y &= (\tilde{T}_{p_0})_{X}Y + (S_{p_0})_{X}Y -(S_{p_0})_{Y}X,\\
(R_{p_0})_{XY} &= (\tilde{R}_{p_0})_{XY} + [(S_{p_0})_Y, (S_{p_0})_X] + (S_{p_0})_{(S_{p_0})_{Y}X-(S_{p_0})_{X}Y},
\end{align*}
where $R_{p_0}$ and $T_{p_0}$ are the curvature and torsion of $\T$ in $p_0$ and  $S_{p_0} \in S^2 V^* \otimes V^*$. We denote by $\mathcal{S}$ the set of homogeneous structures on a Fedosov manifold $(M, \omega, \T)$. Hence, by Theorem \ref{Theorem IS1}, we have the following classification of homogeneous structure tensors in $\mathrm{Sp}(V)$-invariant subspaces:
\begin{equation*}
\mathcal{S} = \mathcal{S}_1 + \mathcal{S}_2 + \mathcal{S}_3,
\end{equation*}
where
\begin{align*}
\mathcal{S}_1&= \{ S \in \mathcal{S}:
\: S_{XYZ} = \omega_{ZY}\omega_{XU}+ \omega_{ZX}\omega_{YU}, U \in \mathfrak{X}(M) \},
\\
\mathcal{S}_2 &= \{ S \in \mathcal{S}:\:
\SC{XYZ}  S_{XYZ} = 0,\, \mathrm{s}_{13}(S)=0\},
\\		
\mathcal{S}_3 &=  \{ S \in \mathcal{S} :\: S_{XYZ}= S_{XZY} \},
\end{align*}
and
\begin{equation*}
	\mathrm{s}_{13}(S)(Z)=\sum _{i=1}^n \left(S_{e_i Z e_{i+n}}-S_{e_{i+n} Z e_{i}}\right),
\end{equation*}
for a symplectic basis $\{ e_1,\ldots,e_n,e_{n+1},\ldots, e_{2n}\}$ of $T_{p_0}M$.

\begin{definition}\label{Definition C3}
	Let $S \in \mathcal{S}$ be a homogeneous Fedosov structure. It is of \emph{type $i$}  if $S$ lies in $\mathcal{S}_i$ and correspondingly is of \emph{type $i+j$} if lies in $\mathcal{S}_i + \mathcal{S}_j$ with $i$, $j \in \{1,2,3\}$ and $i \neq j$.
\end{definition}
Hence, Fedosov homogeneous structure are classified into eight different classes.

\begin{remark}\label{Remark C4}
In \cite{V1985} the author gives a decomposition of the curvature tensor of a symplectic connection in two $\mathrm{Sp}(V)$-irreducible submodules: Ricci type and Ricci flat. Hence, by virtue of \eqref{Equation 16} and Theorem \ref{Theorem AS3}, there is any many as four different classes of symplectic curvature tensor of Fedosov AS-manifolds. We can combine this idea to refine the classification in Definition \ref{Definition C3} to get as many as thirty two different classes of Fedosov AS-manifolds.
\end{remark}

With respect to the classification of homogeneous structures in Definition \ref{Definition C3} and the classification of torsions $\tors$ o AS-manifolds in Definition \ref{Definition C1}, we have the following result which is a consequence of the expression of $A_2$ in \eqref{Equation 15}.

\begin{proposition} \label{Proposition C5}
Let $(M,\omega ,\nabla )$ a Fedosov manifold equipped with homogenous structure $S$.
 \begin{itemize}
 \item If $S\in \mathcal{S}_1$, then the torsion $\tors$ of $\TT=\nabla -S$ belongs to $\tilde{\mathcal{T}}_1$.
 \item If $S\in \mathcal{S}_2$, then the torsion $\tors$ of $\TT=\nabla -S$ belongs to $\tilde{\mathcal{T}}_2$.
 \item If $S\in \mathcal{S}_3$, then the torsion $\tors$ of $\TT=\nabla -S$ vanishes. The manifold $(M, \omega, \TT)$ is a Fedosov manifold with parallel curvature.
 \end{itemize}
 \end{proposition}


\section{Fedosov AS-manifold of linear type.}

Let $(M,\omega , \T)$ be a Fedosov manifold equipped with a AS-homogeneous structure $S$, that is
\[
\TT R=0, \quad \TT \omega =0, \quad \TT S=0,
\]
for $S=\T -\TT$.

\begin{definition}\label{Definiton FL1}
	A homogeneous Fedosov structure $S$ in $(M,\omega , \T)$ is said to be of \emph{linear type} if it belongs to the class $\mathcal{S}_1$, that is
\begin{equation} \label{Equation 17}
	S_XY = \omega(X,Y) \xi - \omega(Y, \xi)X,
\end{equation}
for a vector field $\xi \in \mathfrak{X}(M)$.

\end{definition}

\begin{theorem}\label{Theorem FL2}
	A Fedosov manifold $(M,\omega ,\T)$ admitting a homogeneous structure tensor of linear type does not admit any pseudo-Riemannian metric such that  $S \cdot g = 0$.
\end{theorem}

\begin{proof}
Let $\eta$ be a vector such that $\omega (\eta, \xi) = 1$. Let $g$ be a pseudo-Riemannian metric such that $S\cdot g = 0$, that is
\begin{align*}
0&= g (S_X Y, Z) + g(Y, S_X Z) \\
&=	\omega(X,Y)g(\xi,Z)-\omega(Y,\xi)g(X,Z) + g(Y,\xi)\omega(X,Z)-g(Y,X)\omega(Z,\xi).
\end{align*}
Taking $Y=Z=\eta$ we get $ g(X,\eta) = g(\eta ,\xi)\omega(X,\eta)$. We then get $g(\eta ,X)=0$ for $X = \xi$ and for  $X\in \eta ^\perp = \{ v:\omega (v,\xi)=0\}$, which is impossible since $g$ is not degenerate.
\end{proof}

\begin{remark}
	This last Theorem shows that homogeneous structure tensors on Fedosov manifolds can never be studied under the perspective of Kiri\v{c}enko's Theorem as they are genuine non-metric objects.
\end{remark}

We fix the notation of curvature and torsion tensor fields associated to one connection $\T$,
\begin{align*}
	R_{XY} Z =& \T_{[X,Y]} Z - \T_X(\T_Y Z) +\T_Y (\T_X Z), \\
	T_X Y  =& \T_X Y - \T_ YX - [X,Y],
\end{align*}
for the curvature, we will work both with $(1,3)$-tensors and $(0,4)$-tensors given by the isomorphism,
\begin{equation*}
	R_{XYZU} = \omega (R_{XY}Z, U), \quad X,Y,Z,U \in TM.
\end{equation*}

\begin{proposition}\label{Proposition FL4}
	The condition $\TT S = 0$ is equivalent to $\TT \xi = 0$.
\end{proposition}
\begin{proof}
	Substituting \eqref{Equation 17} in $(\TT_X S)(Y,Z) = \TT_X (S_YZ) - S_{\TT_XY}Z - S_Y(\TT_XZ)$, we get
	\begin{align*}
	(\TT_X S)(Y,Z)	
	&=  \TT_X \left(\omega(Y,Z) \xi - \omega(Z,\xi)Y \right)\\
	&  \hspace{0.25cm} - \omega(\TT_XY,Z)\xi + \omega(Z, \xi) \TT_XY \\
	&  \hspace{0.25cm} - \omega(Y,\TT_XZ)\xi + \omega(\TT_XZ, \xi) Y \\
	&= \hspace{0.03cm} X \left(\omega(Y,Z) \right) \xi + \omega(Y,Z) \TT_X \xi - X \left( \omega(Z,\xi)\right) Y - \omega(Z,\xi) \TT_X Y		\\
	& \hspace{0.25cm}- \omega(\TT_XY,Z)\xi   \hspace{5.25cm}+ \omega(Z, \xi) \TT_XY   \\
	& \hspace{0.25cm}- \omega(Y,\TT_XZ)\xi \hspace{2.55cm} + \omega(\TT_XZ, \xi) Y .
	\end{align*}
	Using $X ( \omega (Y,Z) ) = \omega(\TT_XY,Z) + \omega(Y, \TT_XZ)$ (that is $\TT \omega = 0$), we collect the columns and we get
	\begin{equation*}
		(\TT_X S) (Y,Z) = \omega(Y,Z) \TT_X \xi - \omega(Z,\TT_X\xi)Y.
	\end{equation*}
	If $\TT_X S = 0$, then $\TT_X \xi$ is proportional to any vector field $Y$, and hence $\TT_X \xi = 0$. Conversely, if $\TT_X \xi = 0$, then, we substitute in the previous equation and get $\TT_X S=0$.
\end{proof}

In particular, the vector field $\xi$ defining a homogeneous structure satisfies
\begin{equation}\label{Equation 18}
	\T_X \xi = S_X \xi = \omega(X, \xi)\xi.
\end{equation}

Following the ideas of classifications of homogeneous structures of linear type in the pseudo-Riemannian case (see \cite[Chap. 5]{CC2019}), we study the curvature and Ricci tensor of $\T$.

	\begin{proposition}\label{Proposition FL5}
		Fedosov AS-manifolds of linear type satisfy,
		\begin{align}
		R_{\xi XYZ} &= \omega_{X\xi}\omega_{Y\xi}\omega_{Z\xi} R_{\xi \xp \xp \xp}, \label{Equation 19} \\
		R_{XYUW} &= \left(- \omega_{XY} - \omega_{\xp X} \omega_{Y \xi} + \omega_{\xp Y} \omega_{X\xi} \right) \omega_{U\xi} \omega_{W \xi} R_{\xi \xp \xp \xp} \label{Equation 20}\\
		&\hspace{0.4cm}- \omega_{X \xi} R_{Y \xp U W} - \omega_{Y\xi} R_{\xp XUW} \nonumber
		\end{align}	
		for every $X,Y,U,W \in TM$ and any $\xp\in TM$ such that $\omega(\xp, \xi) = 1$.
	\end{proposition}
	\begin{proof}

From \eqref{Equation 18} and the fact that the torsion of $\nabla$ vanishes, we get
	\begin{align*}
		R_{XY}\xi &= \omega([X,Y],\xi) \xi - \T_X (\omega(Y,\xi) \xi)  + \T_Y (\omega(X,\xi)\xi)\\
		&=  \omega(\T_XY, \xi) \xi - \omega(\T_Y X,  \xi) \xi - X(\omega(Y,\xi)) \xi + Y(\omega(X,\xi)) \xi .
	\end{align*}
As $\T \omega = 0$ and get, this last expression simplifies to $R_{XY}\xi = - \omega(Y, \T_X \xi) \xi + \omega(X, \T_Y \xi) \xi$ which, again by \eqref{Equation 18}, gives
	\begin{equation}\label{Equation 21}
		R_{XY}\xi = 0,
	\end{equation}
together with
	\begin{equation}\label{Equation 22}
		R_{\xi X}Y = R_{\xi Y}X,
	\end{equation}
from the first Bianchi identity.

As $R_{XY} \cdot \omega = 0$, then,
\begin{equation}\label{Equation 23}
	R_{XYZU} = R_{XYUZ}.
\end{equation}

The condition $\TT R = 0$ (that is, $\T_X R = S_X R$) reads
	\begin{equation*}
	(\T_X R)_{YZUW} = - R_{S_XY \,ZUW} - R_{Y\,S_XZUW}- R_{YZ\,S_XUW}- R_{YZU\,S_XW}.
	\end{equation*}
Applying the second Bianchi identity, we get
\begin{align*}
	0 = &\SC{XYZ} \left(- R_{S_XY \,ZUW} - R_{Y\,S_XZUW}- R_{YZ\,S_XUW}- R_{YZU\,S_XW}\right)\\
     = &\SC{XYZ} \left(- \omega_{XY}R_{\xi ZUW}- \omega_{XZ}R_{Y\xi UW}- \omega_{XU}R_{YZ\xi W} - \omega_{XW}R_{YZU \xi}\right.\\
	& \left.  + \omega_{Y \xi} R_{XZUW} +\omega_{Z \xi} R_{YXUW} +\omega_{U \xi} R_{YZXW}+ \omega_{W \xi} R_{YZUX}\right),
\end{align*}
which by virtue of \eqref{Equation 21}, \eqref{Equation 23} and the first Bianchi identity reduces to
	\begin{equation}\label{Equation 24}
	0 = \SC{XYZ} \left(\omega_{XY}R_{\xi ZUW}+ \omega_{ X\xi} R_{YZUW}\right).
	\end{equation}
Choosing $Z=\xi$, we have
	\begin{equation} \label{Equation 25}
	\omega_{X\xi}R_{\xi YUW} = \omega_{Y\xi}R_{\xi XUW}. 
	\end{equation}
	If we choose $X= \xp$ in \eqref{Equation 25}, then, we get $R_{\xi YUW}= \omega_{Y\xi} R_{\xi \xp U W}$, using symmetry of \eqref{Equation 23} and applying equation above we have $R_{\xi YUW}= \omega_{Y\xi} \omega_{U \xi} R_{\xi \xp \xp W}$, and proceeding in an analogous way, using \eqref{Equation 22} and \eqref{Equation 23}, we conclude \eqref{Equation 19}. Substituting $Z = \xp$ in \eqref{Equation 24} and using \eqref{Equation 19}, we get \eqref{Equation 20}.
	\end{proof}

\begin{remark}\label{Remark FL6}
	Equation \eqref{Equation 24} can be refined using \eqref{Equation 19},
	\begin{equation*}
	0 = \SC{XYZ} \left(\omega_{XY}\omega_{Z\xi}\omega_{U\xi}\omega_{W\xi} R_{\xi \xp \xp \xp}+ \omega_{ X\xi} R_{YZUW}\right).
	\end{equation*}
\end{remark}

 We now proceed by parts to prove the main result of this section (see Theorem \ref{Theorem FL10})  which characterize Fedosov AS-manifolds of linear type in terms of a foliation of Hamiltonian leaves of codimension $1$.

\begin{lemma}\label{Lemma FL7}
	The distribution $\mathcal{D}=\{ X\in \mathfrak{X}(M) : \omega(X,\xi) = 0 \}$ is an integrable distribution.
\end{lemma}
\begin{proof}
	Given $X,\,Y \in \mathcal{D}$, we compute,
	\begin{equation*}
	\omega([X,Y],\xi)= \omega(\T_XY, \xi) - \omega(\T_YX, \xi)
	\end{equation*}
	we use that $\T \omega = 0$,
	\begin{equation*}
	\omega([X,Y],\xi)= X(\omega(Y, \xi))- \omega(Y,\T_X \xi) - Y(\omega(X, \xi)) + \omega(X,\T_Y \xi),
	\end{equation*}
	finally, we use that $X,\,Y \in \mathcal{D}$ and \eqref{Equation 18},
	\begin{equation*}
	\omega([X,Y], \xi) = - \omega(X,\xi)\omega(Y,\xi) + \omega(Y,\xi)\omega(X,\xi) = 0.
	\end{equation*}
	Hence, the distribution $\mathcal{D}$ is integrable.
\end{proof}

 \begin{lemma}\label{Lemma FL8}
 	The vector field $\xi$ satisfies the following properties,
 	\begin{enumerate}
 		\item It is a geodesic vector field with respect to $\T$.
 		\item Its flow preserves the symplectic form.
 	\end{enumerate}
 \end{lemma}

 \begin{proof}
 	The first statement comes from \eqref{Equation 18}, that is, $\T_{\xi}\xi = 0$. With respect to the second item, using $\T \omega = 0$ and that $\T$ is torsion free, for $X, Y$ two vector fields
 	\begin{align*}
 		(\mathcal{L}_{\xi}\omega)(X,Y) &=  \omega(\T_{\xi}X,Y) - \omega(\T_{\xi} X, Y) - \omega(X, \T_{\xi} Y)\\
 		& +  \omega(X, \T_{\xi} Y) + \omega(\T_{X}\xi , Y) + \omega(X, \T_{Y} \xi)\\
 		&=  \omega(X, \xi) \omega(\xi , Y) + \omega(X, \xi) \omega(Y, \xi) = 0
 	\end{align*}
 	which means that the flow of $\xi$ preserves the symplectic form.
 \end{proof}

\begin{theorem} \label{Theorem FL10}
Let $(M,\omega,\T)$ be a connected and simply-connected Fedosov manifold endowed with a homogeneous structure $S$ of linear type. Let $\xi$ the vector field associated to $S$. Then:
\begin{itemize}
\item the manifold is foliated by the leaves defined by the Hamiltonian $H$ of the Hamilton equation $i_\xi \omega =dH$,
\item the connection $\T$ restricts to the leaves and in particular, the leaves are totally geodesic submanifolds,
\item the leaves are flat manifolds.
\end{itemize}
Furthermore, if in addition $\TT = \T - S$ is complete, $M$ is Fedosov homogeneous.
\end{theorem}
\begin{proof}
As $M$ is simply-connected, the invariance of $\omega$ with respect to $\xi$ gives the existence of a function $H:M\to\mathbb{R}$ solving the Hamilton equation $i_{\xi} \omega = d H$. For any $X\in\mathcal{D}$, we have $dH(X)=\omega (\xi, X)=0$, so that $H$ is constant along the leaves of $\mathcal{D}$.  For $X,Y$ vector fields tangent to the distribution, $\omega (\T _X Y,\xi)=X(\omega (\xi, Y))+\omega (\T _X\xi,Y)=0$, and the connection restricts to the leaves. Finally, the curvature of $\T$ vanishes along the leaves from \eqref{Equation 20}.
\end{proof}
The local version of this last result is straightforward.

\begin{theorem} \label{Theorem FL11}
Let $(M,\omega,\T)$ be a Fedosov manifold endowed with a homogeneous structure $S$ of linear type. Let $\xi$ the vector field associated to $S$. Then:
\begin{itemize}
\item the manifold is foliated by the leaves locally defined by the Hamiltonian $H$ of the Hamilton equation $i_\xi \omega =dH$,
\item the connection $\T$ restricts to the leaves and in particular, the leaves are totally geodesic submanifolds,
\item the leaves are flat manifolds.
\end{itemize}
Furthermore, $M$ is locally Fedosov homogeneous.
\end{theorem}

We finish with two examples of Fedosov homogeneous structures of linear type.

\begin{example}
Let  $M=\{(x,y)\in \mathbb{R}^2:x>0\}$ be the half-plane endowed with a Fedosov structure defined by the symplectic form $\omega =  \frac{1}{3x^2} \mathrm{d} x \wedge \mathrm{d} y$ and the connection $\nabla$ given by the following non-vanishing Christoffel symbols
\[
\Gamma ^1 _{11} =-\frac{4}{3x},\qquad \Gamma ^2_{12}=\frac{2}{3x},\qquad \Gamma ^2_{21}=-\frac{2}{3x}.
\]
We consider the vector field
\[
\xi= x\frac{\partial}{\partial y},
\]
and the tensor field
\[
S_X Y = \omega_{XY} \xi - \omega_{Y \xi} X.
\]
One can check that
\[
\T \omega = 0, \quad T= 0, \quad \TT R = 0, \quad \TT S = 0,
\]
where $R$ is the curvature of $\nabla$ and $\TT = \T-S$, that is, $S$ is a homogeneous structure of linear type as in Theorem \ref{Theorem FL11}. The manifold $M$ is foliated by the leaves ($\{x =cte\}$) defined by the Hamiltonian $H(x,y) =\frac{-1}{3} \log (x)$, the connection $\T$ restricts to the leaves, and they are totally geodesic and flat.  Furthermore, since the vector fields $x\partial/\partial x$ and $\xi=x\partial/\partial y$ are complete geodesic vector fields of $\TT$, this connection is complete and $M$ is homogeneous Fedosov manifold as in Theorem \ref{Theorem FL10}. With respect to the group acting transitively on $M$, it turns our that $\curv=0$. The Nomizu construction \eqref{Nomizu construction} gives $M=G$, with $G=\mathrm{Aff}(1)_0$ the connected component of the identity of the group of affine transformation of $\mathbb{R}$.

If, instead, we consider
\[
\Gamma ^1 _{11} =-\frac{2}{x},
\]
as the only non-vanishing Christoffel symbol, and (for the sake of convenience with the computations) $\omega = \frac{1}{x^2} \mathrm{d} x \wedge \mathrm{d} y$, we again get that $(M,\omega , \nabla)$ is a homogeneous Fedosov manifold of linear type as in Theorem \ref{Theorem FL10} with $\xi =x\partial /\partial y$. In this case, one can check that
\[
\curv _{\xi \eta}\eta=-2\xi, \qquad \curv _{\xi \eta}\xi=0,
\]
with $\eta =x\partial /\partial x+ y\partial /\partial y$. The algebra $\mathfrak{g}=\mathrm{span}\{\xi, \eta, \curv _{\xi \eta}\}$ built from the Nomizu construction is the Lie algebra of the three dimensional Lie group acting on $M$ such that $M=G/H$ with $H\simeq \mathbb{R}$. One checks that, following the convention of \cite[p. 2194]{B2001}, the Lie algebra $\g$ of $G$ is the Lie algebra of type Bianchi VI, with real parameter $h=2$.

\end{example}


\addcontentsline{toc}{section}{References}

\nocite{*} 

\bibliographystyle{plain}

\end{document}